\documentclass[pdflatex,sn-mathphys-num]{sn-jnl}% Math and 

\usepackage{graphicx}%
\usepackage{multirow}%
\usepackage{amsmath,amssymb,amsfonts}%
\usepackage{amsthm}%
\usepackage{latexsym}
\usepackage[title]{appendix}%
\usepackage{xcolor}%
\usepackage{textcomp}%
\usepackage{manyfoot}%
\usepackage{booktabs}%
\usepackage{algorithm}%
\usepackage{algorithmicx}%
\usepackage{algpseudocode}%
\usepackage{listings}%
\usepackage{caption}
\usepackage{subcaption}
\usepackage[normalem]{ulem}
\usepackage{tikz}
\usetikzlibrary{calc}
\usetikzlibrary{arrows.meta, positioning, shadows}

\def \bx{{\mathbf {x}}}
\def \by{{\mathbf {y}}}
\def \bv{{\mathbf {v}}}
\def \be{{\mathbf {e}}}
\def \bw{{\mathbf {w}}}

\theoremstyle{thmstyleone}%
\newtheorem{theorem}{Theorem}%  meant for continuous 
\newtheorem{proposition}[theorem]{Proposition}% 
\theoremstyle{thmstyletwo}%
\newtheorem{remark}{Remark}%

\theoremstyle{thmstylethree}%

\begin{document}

\title[Article Title]{Optimal mean-time path planning for unmanned underwater vehicles: a Hamilton-Jacobi approach}

\author[1]{\fnm{Jonathan} \sur{Valyou}}\email{jmv18b@fsu.edu}

\author[2]{\fnm{Jeremy} \sur{Brandman$^\dagger$}\footnotetext{$\dagger$ Work conducted while affiliated with the U.S. Naval Research Laboratory.}}\email{brandmanj@metsci.com}

%\footnote{Work conducted while affiliated with the U.S. Naval Research Laboratory.  Currently affiliated with Northrop Grumman Corporation; email: jeremy.brandman@ngc.com}

\author*[1]{\fnm{Sanghyun} \sur{Lee}}\email{slee17@fsu.edu}

\affil[1]{\orgdiv{Department of Mathematics}, \orgname{Florida State University}, \orgaddress{\street{1017 Academic Way}, \city{Tallahassee}, \postcode{32306}, \state{Florida}, \country{United States}}}

\affil[2]{\orgname{Metron, Inc.}, \orgaddress{\street{1818 Library Street, Suite 600}, \city{Reston}, \state{Virginia}, \postcode{20190}}}

%\affil[3]{\orgdiv{Department}, \orgname{Organization}, \orgaddress{\street{Street}, \city{City}, \postcode{610101}, \state{State}, \country{Country}}}

\abstract{
Unmanned underwater vehicles (UUV) integrate ocean forecasts with path planning algorithms in order to identify energy- or time-minimizing paths that enable mission completion.  Typically, a well-defined deterministic ocean forecast is assumed to be available for path planning; however, in practice, different ocean forecasts can disagree.  In this paper, we extend previous work on deterministic optimal path planning \cite{Brandman2023} to identify optimal mean-time paths when presented with an ensemble of possible ocean forecasts.  In particular, we formulate a system of time-independent Hamilton-Jacobi partial differential equations that incorporates forecast uncertainty and yields the optimal mean reachability travel time and the necessary controls to find the associated optimal path.  An efficient numerical solution of this system of PDEs is obtained through an extension of the Fast Sweeping Method~\cite{Kao2004}; verification and benchmarking results are provided.  Additional numerical examples illustrate the impact uncertainty can have on the optimal path; in particular, these results demonstrate that the vehicle's optimal path can deviate significantly from the deterministic optimal paths associated with the individual ensemble members.
}




\maketitle

\section{Introduction}

Unmanned underwater vehicles (UUV) are of interest because of their ability to navigate the ocean floor and perform important search and surveillance tasks without direct human intervention.  At present, mission duration for such vehicles is limited by battery life; as a result, path planning is commonly used to identify an energy-minimizing or time-optimal path based on a forecast of the ocean current.  However, the prediction of the ocean's local velocity field is typically uncertain~\cite{Lermusiaux2006}, resulting in the divergence of different ocean current models in a given domain of interest.  This raises the following question: how to identify an optimal vehicle path, given that the underlying ocean current model is uncertain?

Path planning problems can often be solved effectively using techniques from optimal control.  In this context, the solution of the Hamilton-Jacobi-Bellman partial differential equation (PDE) associated with a given optimal control problem provides a direct means of computing the minimum cost to reach a target state from an initial one~\cite{Bellman1957, Fleming2006}.  The Hamilton-Jacobi-Bellman equation represents a profound and non-intuitive link between a family of optimization problems and the solution of a nonlinear PDE.   The simplest example where Hamilton-Jacobi-Bellman equations apply to path planning is the eikonal equation, which can be used to identify shortest paths in the Euclidean metric~\cite{Sethian1999, Tsitsiklis1995}.  

The original work addressing path planning under uncertainty built upon the principles of stochastic optimal control and robust control.  In early algorithms, this amounted to including a diffusive term in the Hamilton-Jacobi-Bellman equation in order to account for uncertain dynamics~\cite{Isaacs1965}.  This led to the concept of robust control theory, where worst-case solutions are considered~\cite{Basar1995}.  Ensemble-based methods represent an alternative approach; these methods generate multiple realizable dynamics through the application of deterministic processes to a family of possible scenarios~\cite{Evensen2003, Calafiore2006, Mesbah2016}.  Our approach falls under the latter category. 

A wide range of approaches have been proposed for deterministic vehicle path planning, including graph-based approaches such as the $A^*$ algorithm~\cite{Garau2005}, level set-based methods~\cite{Lolla2014}, methods based on fast marching~\cite{Petres2007} and fast sweeping methods~\cite{Brandman2023, Parkinson2020}.  Approaches to incorporating uncertainty in path planning are less mature and are typically extensions of deterministic approaches.  These include graph-based frameworks for path planning with time dependent uncertainty~\cite{Wellman2013, Rathbun2002}, and level set path planning with an added stochastic term~\cite{Mitchell2005, Subramani2018} or risk cost function additions~\cite{Subramani2019}.  Graph-based methods utilize a modified exhaustive search algorithm, such as Dijkstra's Algorithm, to store all possible paths and select the one that minimizes the reachability time. The level set approach uses Monte Carlo to evolve in time an ensemble of possible optimal paths; the path that arrives first is recovered by backtracking or solving a constructed Boundary Value Problem backwards in time.  While these methods are robust and offer convergence guarantees, their applicability is limited by how accurately the stochastic ocean models they rely on reflect real-world ocean forecasting.

 In this paper, we develop a relatively straightforward approach to vehicle path planning under uncertainty that only requires, as input, an ensemble of ocean forecast models and their associated likelihoods.  Specifically, we present a novel approach to optimal path planning under uncertainty utilizing a system of Hamilton-Jacobi partial differential equations to identify a vehicle path that globally minimizes the vehicle's mean reachability time.  While the resulting formulation is elegant and relatively straightforward, there is a significant limitation: it does not address the worst-case behavior of the proposed path, as would be identified using robust control theory.  Instead our method produces a result that is similar to a risk-neutral control problem as the optimal path based on the mean reachability time is indeed constructed under a risk-neural setting.  Additionally, our approach differs from standard methods as we also recover the deterministic optimal paths under each ocean model ensemble member.

The contributions of this work are twofold.  First, from an application perspective, the present work extends the deterministic optimal path planning framework introduced in~\cite{Brandman2023} through the incorporation of an ensemble of potential ocean current velocity models, each weighted according to its likelihood of realization. In particular, we demonstrate that the framework is practical for path planning under realistic ocean models.  Second, from a numerical analysis perspective, the present work generalizes the Lax Friedrich's Fast Sweeping method~\cite{Kao2004} - introduced as an algorithm for solving a single Hamilton-Jacobi partial differential equation - to systems of Hamilton-Jacobi partial differential equations.  The novelty lies in the generalized fast sweeping algorithm, proposed here, for solving systems of Hamilton-Jacobi equations.

The remainder of the paper is organized as follows. Section \ref{sec:2} uses the dynamic programming principle to derive the system of Hamilton-Jacobi PDEs governing the optimal mean-time for the vehicle to reach its destination.  In addition, we show how to identify optimal vehicle controls; these are used to determine the corresponding optimal mean-time path.  
The proposed numerical method for solving the system of Hamilton-Jacobi equations, based on alternating sweeps for each equation, is presented in Section \ref{sec:3}.  Section \ref{sec:4} presents two-dimensional numerical results that validate our approach and highlight cases where the mean-optimal path deviates significantly from the optimal paths associated with the individual ensemble members.  Section \ref{sec:5} summarizes our results and discusses future directions.

\section{Governing System: Uncertainty Framework}
\label{sec:2}

\subsection{Transport Model}

We model UUV transport as the combination of its own propulsion and advection by the background ocean current.
For a given computational domain $\bx \in \Omega  \subset \mathbb{R}^2$ and prescribed time interval $t \in [0,T_f]$ where $T_f$ is the total travel-time, we define $\by :=\by(t): (0,T_f] \rightarrow \Omega$ as the vehicle's position, $\bv_w:=\bv_w(t): [0,T_f] \rightarrow \mathbb{R}^d$ as the vehicle's velocity relative to water (propulsion), and $\bv_c:= $ $\bv_c(\bx,t):\Omega \times [0,T_f] \rightarrow \mathbb{R}^d$ as the ocean current velocity model. Applying the approach presented in~\cite{Brandman2023}, $\by(t)$ satisfies the following ordinary differential equation (ODE):
\begin{equation}
   \frac{d\by}{dt} = \bv_w(t) + \bv_{c}(\by,t) \text { in } \Omega \times (0,T_f],
\label{eqn:det-ode}
\end{equation}
where the initial condition is given as 
$\by(0) = \bx_S$.

We assume the maximum ocean current speed ($\max\limits_{\bx,t}$$|\bv_c(\bx,t)|$) is smaller than the maximum vehicle speed, $s_{max}\in \mathbb{R}^+$.  This ensures that the vehicle can reach any target point $\bx_E$ in the domain. We express these constraints as

\begin{equation}
{\max\limits_{t}|\bv_w(t)|  \leq } s_{max} \quad  \forall t \quad  \text{ and } \quad 
{ \max\limits_{\bx,t}}|\bv_{c}(\bx,t)| { < } s_{max} \quad \   \forall (\bx,t).  
\label{eqn:transport_ode_condition}
\end{equation}

In particular, the parametrized vehicle's position $\by(t)$ is denoted as a given path $\gamma:=\gamma(t)$ from $\gamma(0) = \bx_S$ to $\gamma(T) = \bx_E$. Let $\bx_p \in \Omega$ be an arbitrary point along the path $\gamma$ that is reached at a time $t_p$.  See Figure \ref{fig:detpathplanningproblem} for more details.

\begin{figure}[!ht]
\centering

\begin{tikzpicture}

  \draw[thick] (-1,0) rectangle (8,4);

  \coordinate (X_s) at (1,1);
  \coordinate (X_E) at (7,3);

  \fill (X_s) circle (2pt) node[below] {$\by(0)=\bx_s$};
  \fill (X_E) circle (2pt) node[above right] {$\bx_E$};

  \draw[thick, blue, ->] 
    (X_s) .. controls (2,3) and (6,1) .. 
    node[midway, above] {$\gamma$}
    (X_E);

  \coordinate (Xp) at (1.75,1.7);
  \fill (Xp) circle (2pt) node[above left] {$\by(t_p) = \bx_p$};

  \draw[->, thick, brown] (Xp) -- ++(1,0) node[midway, below] {$\vec{\bv}_w$};

  \coordinate (vwTip) at ($(Xp)+(1,0)$);
  \draw[->, thick, green!60!black] (vwTip) -- ++(0.8,0.3) node[midway, below] {$\vec{\bv}_c$};

  \draw[->, thick, red] (Xp) -- ++(1.8,0.3) node[midway, above] {$\frac{d\by}{dt}$};

\end{tikzpicture}
\caption{Deterministic path planning}
\label{fig:detpathplanningproblem}
\end{figure}
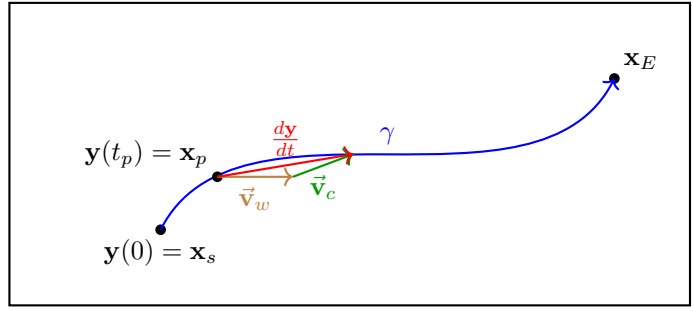

\subsection{Path planning under uncertainty}
In practice,  ocean forecasts carry significant uncertainty; this introduces variability in the ocean current velocity $\bv_c(\bx,t)$. To address this, we aim to incorporate an ensemble of $n$ ocean current velocity models into our path planning algorithm and obtain an optimal path $\gamma(t)$ that minimizes the mean travel-time of the vehicle.

We denote by $\bv_c(\bx,t)$ the true ocean current velocity model (Figure \ref{fig:ensemble}) (a); we assume that we have access to an ensemble of $n$ ocean current velocity models $\bv_{c,1}(\bx,t), \bv_{c,2}(\bx,t), ..., \bv_{c,n}(\bx,t)$ corresponding to different forecasts of $\bv_c(\bx,t)$ (Figure \ref{fig:ensemble}) (b).
%To address this, we consider  
We note that each $\bv_{c,i}$ is associated with the probability  $p(i) \in [0,1]$ where $i=1,2,...,n$ of being realized such that $\sum_{i=1}^n p(i) = 1$.  Moreover, the formed probability distribution indicates our certainty of how likely each $\bv_{c,i}$ in relation to one another is to represent the true ocean model.  The impact will be seen later in Section \ref{sec:mrt} in our definition of the mean reachability time.

\begin{figure}[!ht]
\centering
\begin{tikzpicture}[scale=0.7, >=Stealth, every node/.style={font=\small}]

\draw[thick] (0,0) rectangle (3,3);
\node at (1.5,3.3) {\(\mathbf{v}_c(\bx,t)\)};
\node at (1.5,-0.5) {(a) true};

\foreach \x in {0.5,1.5,2.5} {
  \foreach \y in {0.5,1.5,2.5} {
    \pgfmathsetmacro{\dx}{0.3*sin(\y*30)}
    \pgfmathsetmacro{\dy}{0.3*cos(\x*30)}
    \draw[->, blue] (\x,\y) -- ++(\dx,\dy);
  }
}

\draw[->, thick] (3.2,1.5) -- (4.8,1.5);

\foreach \i/\label in {0/{\(\mathbf{v}_{c,1}(\bx,t)\)}, 3.6/{\(\mathbf{v}_{c,2}(\bx,t)\)}, 9.2/{\(\mathbf{v}_{c,n}(\bx,t)\)}} {
  \draw[thick] (5+\i,0) rectangle (8+\i,3);
  \node at (6.5+\i,3.3) {\label};
  \foreach \x in {0.5,1.5,2.5} {
    \foreach \y in {0.5,1.5,2.5} {
      \pgfmathsetmacro{\dx}{0.3*sin(\y*30 + 5*\i + 10*\x)}
      \pgfmathsetmacro{\dy}{0.3*cos(\x*30 + 5*\i - 10*\y)}
      \draw[->, red!70!black] (5+\i+\x,\y) -- ++(\dx,\dy);
    }
  }
}

\node at (12.8,1.5) {\Large \(\cdots\)};

\node at (11.0,-0.5) {(b) ensemble};

\end{tikzpicture}
\caption{Uncertainty in ocean models: (a) the true ocean current model; (b) the ensemble of forecasted ocean current models.}
\label{fig:ensemble}
\end{figure}
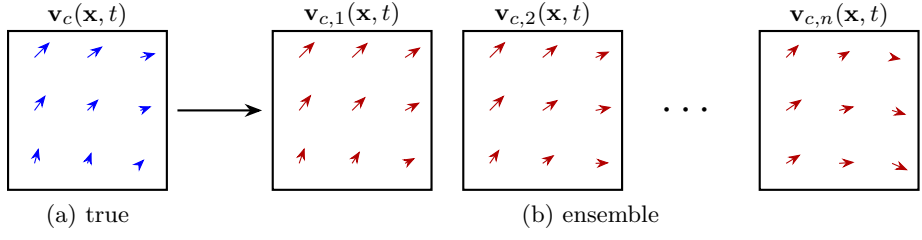

Thus, by considering each ocean model we obtain the following vehicle transport model to obtain $\by_i(t)$, satisfying:
\begin{equation}
   \frac{d\by_i}{dt} = \bv_w(t) + \bv_{c,i}(\bx,t) \text { in } \Omega \times (0,T_f].
\label{eqn:det-ode_i}
\end{equation}
with $\by_i(0) = \bx_S$.
As before, we assume that
\begin{equation}
{ \max\limits_{t}}|\bv_w(t)| { \leq } s_{max}  \ \forall t,  \text{ and } 
{ \max\limits_{\bx,t}}|\bv_{c,i}(\bx,t)| { < } s_{max} \ \forall (\bx,t), 
i=1, \cdots, n
\label{eqn:ode_condition}
\end{equation}
to ensure that the vehicle can reach any target point ${\bf x}_E$ in the domain.

Our approach to path planning in the context of an ensemble of ocean models generalizes the deterministic framework~\cite{Brandman2023}.  For a given starting point $\bx_S$ and end point $\bx_E$, we seek a vehicle path $\gamma$ that connects the two points and minimizes the mean travel-time.  See Figure \ref{fig:meanpathplanningproblem} for an illustration of path planning under uncertainty.

\begin{figure}[!h]
\centering
\begin{tikzpicture}[scale=0.7, >=Stealth, every node/.style={font=\small}]

\foreach \i/\label in {0/{\(\mathbf{v}_{c,1}(\bx,t)\)}, 3.6/{\(\mathbf{v}_{c,2}(\bx,t)\)}, 9.2/{\(\mathbf{v}_{c,n}(\bx,t)\)}} {
  \draw[thick] (5+\i,0) rectangle (8+\i,3);
  \node at (6.5+\i,3.3) {\label};
  
  \foreach \x in {0.5,1.5,2.5} {
    \foreach \y in {0.5,1.5,2.5} {
      \pgfmathsetmacro{\dx}{0.3*sin(\y*30 + 5*\i + 10*\x)}
      \pgfmathsetmacro{\dy}{0.3*cos(\x*30 + 5*\i - 10*\y)}
      \draw[->, red!70!black] (5+\i+\x,\y) -- ++(\dx,\dy);
    }
  }
  
  \draw[thick, black] 
    (5+\i+0.3,0.2) 
    .. controls (5+\i+1.5,1+\i*0.2) and (5+\i+2.5,2+\i*0.1) 
    .. (5+\i+2.7,2.8);
  \filldraw[black] (5+\i+0.3,0.2) circle (2pt);
  \filldraw[black] (5+\i+2.7,2.8) circle (2pt);
  \node at (5+\i+0.9,0.3) {\(x_S\)};
  \node at (5+\i+2.1,2.7) {\(x_E\)};
}

\node at (12.8,1.5) {\Large \(\cdots\)};

\node at (11.0,-0.5) {(a) deterministic optimal paths};

\draw[->, thick] (10.0,-0.8) -- (10.0,-1.8);

\draw[thick] (8.5,-5) rectangle (11.5,-2);
\node at (10.0,-5.5) {(b) mean optimal path};

\foreach \x in {0.5,1.5,2.5} {
  \foreach \y in {0.5,1.5,2.5} {
    \pgfmathsetmacro{\dx}{0.3*sin(\y*30)}
    \pgfmathsetmacro{\dy}{0.3*cos(\x*30)}
    \draw[->, blue!50!gray] (8.5+\x,-5+\y) -- ++(\dx,\dy);
  }
}

\draw[thick, red]
  (8.8,-4.8) 
  .. controls (9.5,-4.0) and (10.5,-3.2)
  .. (11.2,-2.2);
\filldraw[red] (8.8,-4.8) circle (2pt);
\filldraw[red] (11.2,-2.2) circle (2pt);
\node at (9.4,-4.7) {\(x_S\)};
\node at (10.6,-2.3) {\(x_E\)};

\end{tikzpicture}

\caption{Path planning with uncertainty: (a) for each ocean model in our ensemble, deterministic path planning identifies a path that minimizes the vehicle's travel-time; (b) the optimal mean-time path from $x_S$ to $x_E$ minimizes the mean travel-time of the vehicle; this path is potentially distinct from the family of deterministic optimal paths displayed in (a).}
\label{fig:meanpathplanningproblem}
\end{figure}
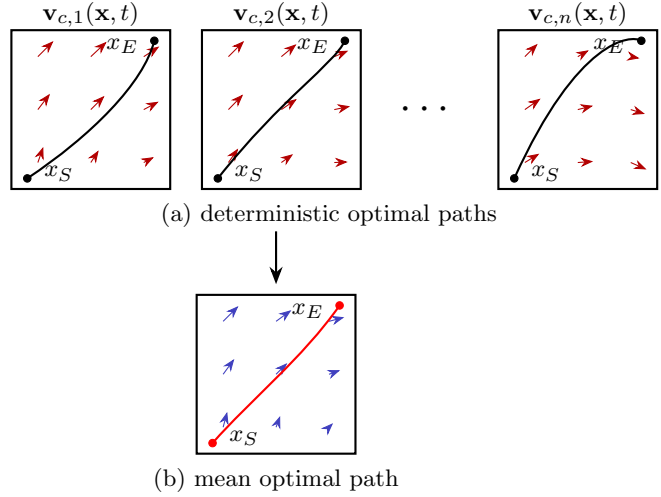

\subsection{Mean Reachability Time}
\label{sec:mrt}
We seek to determine the curve $\gamma(t)$ that minimizes the mean vehicle travel-time between a fixed starting point $\bx_S$ and an arbitrary target point $\bx$ in $\Omega$.  This allows for the computation of multiple possible target points rather than a single $\bx_E$.  Computing such a path directly, through nonlinear optimization, has several drawbacks: it can be computationally expensive (especially if multiple target points are desired) and can stall at non-optimal critical points~\cite{Sethian2003}.

We take a different approach that utilizes principles of optimal control and Hamilton-Jacobi partial differential equations.  First, we use the dynamic programming principle to compute the mean travel time (henceforth referred to as the mean reachability time) as a function of an arbitrary target point residing in our computational domain.  The key to doing this is the observation, presented in Section \ref{subsec:HJ}, that the mean reachability time satisfies a system of Hamilton-Jacobi partial differential equations.  Second, we show that the determination of optimal paths reduces to the solution of ordinary differential equations once the mean reachability time function is computed.

For a given curve $\gamma(t)$ connecting $\bx_S$ to $\bx$, let $T_i := T_i(\bx,\bx_S,\gamma)$ denote the time required by the UUV to complete its transit under the \(i\)‑th ocean current velocity model \(\bv_{c,i}\) while following $\gamma$.  {Here, $T_i(\bx,\bx_S,\gamma)$ is assumed to be continuous and smooth away from $\bx_S$.}  Then the mean reachability time \(\bar{T}(\bx,\bx_S,\gamma)\)  in the presence of \(n\) such models is defined by
\begin{equation}\label{eq:weighted-time}
\bar{T}(\bx,\bx_S,\gamma)
:=\sum_{i=1}^n p(i)\,T_i(\bx, \bx_S,\gamma).
\end{equation}
Our goal is to determine the minimal mean reachability time $u:=u(\bx,\bx_S):\Omega \times \Omega \rightarrow \mathbb{R}$:
\begin{equation}\label{eq:u-def}
u(\bx,\bx_S) := \min_{\gamma} \bar{T}(\bx, \bx_S, \gamma)
\end{equation}
under the $n$ ocean velocity models from a starting position $\bx_S$ to any position $\bx$ in $\Omega$ where the { mean reachability time} minimizing path is denoted $\gamma^*$.  {Here, in theory, $u(\bx,\bx_S)$ is also assumed to be of a continuous and smooth away from $\bx_S$ regularity class.}

Because our novel approach to computing the mean reachability time is based on the dynamic programming principle (and hence recursive), it turns out that we must solve for $u(\bx,\bx_S)$ at all points $\bx \in \Omega$. While this may seem unappealing, this approach provides the ability to pick any $\bx \in \Omega$ as $\bx_E$ and find its { mean reachability time} minimizing path $\gamma^*$.  We will show in Section \ref{sec:3} that the computation of $u(\bx,\bx_S)$ can be carried out efficiently by the fast sweeping method.

\subsection{Maximum UUV Speed}
\label{subsec:max}

\begin{figure}[!h]
\centering
\begin{tikzpicture}[scale=1, every node/.style={font=\small}, >=Stealth]

\draw[thick] (0,0) rectangle (8,4);

\coordinate (X_s) at (1,1);
\coordinate (X_E) at (7,3);
\fill (X_s) circle (2pt) node[below left] {$\bx_s$};
\fill (X_E) circle (2pt) node[above right] {$\bx$};

\draw[thick, blue, ->] 
  (X_s) .. controls (2,3) and (6,1) .. 
  node[midway, above] {$\gamma$}
  (X_E);

\coordinate (Xp) at (1.75,1.7);
\fill (Xp) circle (2pt) node[below ] {$\bx_p$};

\draw[dashed] ($(Xp)+(-0.2,-0.2)$) rectangle ($(Xp)+(0.2,0.2)$);

\begin{scope}[shift={(10,0)}]

  \draw[thick] (0,0) rectangle (3,3);
  \node at (1.5,3.2) {Zoom at $\bx_p$};

  \draw[thick, blue, ->]
    (0.0,0.2) .. controls (1.1,1.5) and (1.9,1.7) .. (3.0,2.5);

  \coordinate (XpM) at (1.5,1.5);
  \fill (XpM) circle (2pt) node[below] {$\bx_p$};

  \draw[->, thick] (XpM) -- ++(1,0) node[right] {$x_1$};
  \draw[->, thick] (XpM) -- ++(0,1) node[above] {$x_2$};

  \draw[->, thick, red] (XpM) -- ++(0.8,0.6);

  \draw (XpM) ++(0.5,0) arc[start angle=0, end angle=50, radius=0.4];
  \node at ($(XpM)+(0.7,0.25)$) {$\theta_p$};

\end{scope}

\draw[dashed] (Xp) -- (11.5,1.5); 

\end{tikzpicture}

\caption{Travel direction setup in 2D}
\label{fig:setup}
\end{figure}
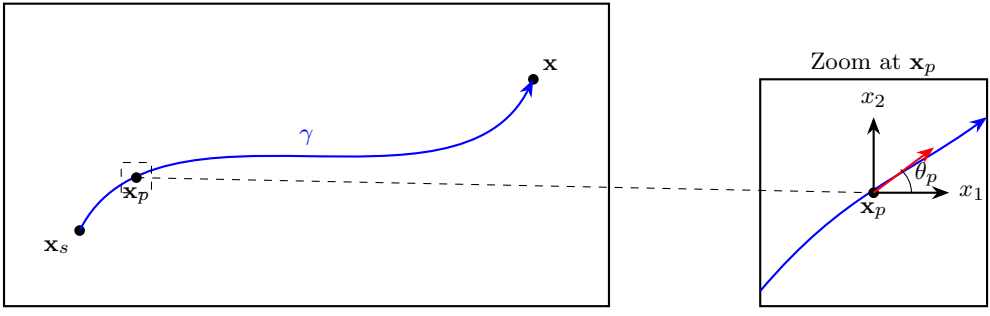

We define $\theta := \theta(\bx,t): \Omega \times [0,T_f] \rightarrow [0,2\pi)$ as the counterclockwise direction the UUV is traveling at a given point $\bx$ with respect to the $x_1$-axis.  Thus, $\theta_p:= \theta(\bx_p,t_p)$ is the UUV's travel direction at spatial position $\bx_p$ and given time $t_p$.  See Figure \ref{fig:setup} for the illustration.

Due to the $\bv_{c,i}$ vectors differing at a point $\bx$, the single $\theta$ direction requires the vehicle to have different maximum speeds under each $\bv_{c,i}$ model.  Therefore, in this paper, we define $s_{i,max}(\theta,\bx,T_i): [0,2\pi] \times \Omega \times [0,T_f] \rightarrow \mathbb{R}$ as the maximum UUV speed relative to the ground in the direction of $\theta$ considering the $i$-th ocean model as the following
\begin{equation}
s_{i,max}(\theta,\bx,T_i) := (\bv_{c,i}(\bx,T_i) \cdot \be_\theta) + \sqrt{s^2_{max} - (\bv_{c,i}(\bx,T_i) \cdot \be_\theta^\perp)^2},
\label{eqn:speed}
\end{equation}
where $\be_\theta := [\cos(\theta),\sin(\theta)]^\top$ is the unit vector in the direction of $\theta$.

To derive the above $s_{i,max}(\theta, \bx, T_i)$, we first view the UUV velocity with respect to the $i$-th ocean model as an orthogonal projection
\begin{equation}
s_{i,max}(\theta,\bx,T_i) \be_{\theta} = \beta_1 \be_\theta + \beta_2 \be_\theta^\perp + (\bv_{c,i} (\bx,T_i) \cdot \be_\theta)\be_\theta + (\bv_{c,i}(\bx,T_i)\cdot \be_\theta^\perp)\be_\theta^\perp.
\end{equation}

To find $\beta_1$ and $\beta_2$, we first observe $\beta_1^2 + \beta_2^2 \leq s_{max}^2$ since the vehicle cannot travel faster than its maximum speed. Second, note that orthogonality implies that $\beta_2 = -(\bv_{c,i}(\bx,T_i) \cdot \be_\theta^\perp)$.
Then, for $\beta_1$, we obtain the following 
\begin{equation}
\beta_1 \leq \sqrt{s_{max}^2-\beta_2^2} =\sqrt{s_{max}^2 -(\bv_{c,i}(\bx,T_i) \cdot \be_\theta^\perp)^2}.
\end{equation}
Finally,  we have derived 
the maximum vehicle speed (relative to the ground) in the $\theta$ direction considering the $i$-th ocean model as
$$
s_{i,max}(\theta,\bx,T_i) = (\bv_{c,i}(\bx,T_i) \cdot \be_\theta) + \sqrt{s^2_{max} - (\bv_{c,i}(\bx,T_i) \cdot \be_\theta^\perp)^2}.
$$

Given the definition of $s_{i,max}(\theta,\bx,T_i)\be_\theta$ as the UUV velocity in the direction of $\theta$ considering the $i$-th ocean model relative to the ground, we can rewrite the Transport ODE \eqref{eqn:det-ode} associated with each ocean model:
\begin{equation}
    \frac{d\by_i}{dt} =  s_{i,max}(\theta,\by_i,T_i)\be_{\theta}
    \label{eq:ODErewritten}
\end{equation}
with initial condition given as 
\begin{equation}
    \by_i(0) = \bx_S,
    \label{eq:ODE_ic_rewritten}
\end{equation}
where $\theta$ is an optimal control in this problem.  We note that, later, an initial condition will correspond to a reachability time $T_i$, giving the vehicle's position at the target location $\bx_E$.

%We note that 
%\begin{equation}
%\max\limits_{\theta, {\bf x},T_i} s_{i,max}(\theta,\bx, %T_i) \leq s_{max}  , \quad \text{for } i=1,...,n
%\end{equation}
%to ensure that the vehicle can reach any target $\bx_E$ %along any path $\gamma$ within the domain.  
We assume throughout the remainder of the paper that 
\begin{equation}
    s_{i,max}(\theta,\bx, T_i) > 0, \quad \text{for } i=1,...,n.
\end{equation}
This ensures that each point in the domain is reachable and avoids degeneracies in the PDE system we will define next.  The above equation is equivalent to the condition stated earlier in the paper:
\begin{equation}
\lVert v_c \rVert < s_{max}.
\end{equation}
%as this can result in degeneracies in the PDE system we %will next define.

\subsection{Hamilton-Jacobi Formulation of the Optimal Mean Reachability Time}
\label{subsec:HJ}

In this section, we identify the system of Hamilton-Jacobi PDEs satisfied by $u(\bx,\bx_S)$ and $T_i(\bx,\bx_S,\gamma^*)$.  The derivation of this system is based on the dynamic programming principle and follows the presentation in~\cite{Brandman2023}.  This result is foundational for our approach; our numerical method for computing $u(\bx,\bx_S)$ and the associated optimal paths $\gamma^*$ is based on a numerical discretization of this system of PDEs.
\\

\begin{proposition}
The minimum mean reachability time $u(\bx,\bx_S)$ satisfies the following system of Hamilton-Jacobi Partial Differential Equations:
%%%
\begin{align}
    & \min_{\theta \in  [0, 2\pi)} \{ -\nabla u \cdot s_{1,max} (\theta, \bx, T_1) \be_\theta + \sum_{i=1}^{n} p(i) \frac{s_{1,max} (\theta, \bx, T_1)}{s_{i,max} (\theta, \bx, T_i)} \} = 0 \label{eqn:unc-hj}\\
    & \nabla T_i(\bx, \bx_S,\gamma^*) \cdot s_{i,max}(\theta^*,\bx, T_i)\be_{\theta^*} = 1 , \quad \text{for } i=1,...,n
    \label{eqn:det-hj}
\end{align}
subject to the boundary conditions
\begin{align}
    & u(\bx_S, \bx_S) = 0 , \quad T_i(\bx_S, \bx_S, \gamma^*) = 0, \quad \text{for } i=1,...,n \label{eqn:bc-hj}
\end{align}
%%%
where the Hamiltonian minimizer is denoted $\theta^*$.
\end{proposition}

\begin{proof}

    Introduce a small change $dS$ in the optimal path connecting the starting point $x_S$ to another point $x$.  It follows from the dynamic programming principle and our definition of $u(\bx,\bx_S)$ that
    \begin{align*}
        &u(\bx,\bx_S) = \min\limits_{\substack{dS}}\{ u(\bx-dS,\bx_S) +  u(\bx,\bx-dS)\}.
    \end{align*}

    Define $0<t_1<t_2<...<t_{p-1}<t_p<...<T_f$ with uniform time-step size $\Delta t=t_p-t_{p-1}$.  Let $dS=\Delta t \cdot s_1(\theta,\bx,T_1)\be_\theta$ where $\Delta t \ll 1$; it follows that
    
    \begin{align*}
        &u(\bx,\bx_S) = \min\limits_{\substack{\theta \in [0,2\pi) \\ 0 {\color{violet} <} s_1(\theta,\bx,T_1) \leq s_{1,max}(\theta, \bx,T_1)}} \{ u(\bx-\Delta t \cdot s_1(\theta,\bx,T_1)\be_\theta, \bx_S) + u(\bx,\bx-\Delta t \cdot s_1(\theta,\bx,T_1)\be_\theta) \}.
    \end{align*}
    By applying a Taylor Expansion to the first term on the right-hand side, we obtain
    \begin{equation}
        u(\bx,\bx_S) = \min\limits_{\substack{\theta \in [0,2\pi) \\ 0 {\color{violet} <}  s_1(\theta,\bx,T_1) \leq s_{1,max}(\theta, \bx,T_1)}} \{ u(\bx,\bx_{S})- \nabla u(\bx,\bx_{S}) \cdot \Delta t \cdot s_1(\theta,\bx,T_1)\be_\theta + u(\bx,\bx-\Delta t \cdot s_1(\theta,\bx,T_1)\be_\theta)\}.
    \end{equation}
    After canceling $u(\bx,\bx_S)$ from each side, we get
    \begin{equation}
        0 = \min\limits_{\substack{\theta \in [0,2\pi) \\ 0 {\color{violet} <}  s_1(\theta,\bx,T_1) \leq s_{1,max}(\theta, \bx,T_1)}} \{ - \nabla u \cdot \Delta t \cdot s_1(\theta,\bx,T_1) \be_\theta + u(\bx,\bx-\Delta t \cdot s_1(\theta,\bx,T_1)\be_\theta)\}.
    \end{equation}

    Next, since $\Delta t\ll1$, $\Delta t \cdot s_1(\theta,\bx,T_1)\be_\theta$ is small.  Therefore, it follows that the optimal path from $\bx-\Delta t \cdot s_1(\theta,\bx,T_1)\be_\theta$ to $\bx$ is approximately straight in the direction $\be_\theta$ with a length of $\Delta t \cdot s_1(\theta,\bx,T_1)$.  Additionally, note that the velocity under the $i$-th ocean model would be $s_i(\theta,\bx,T_i)\be_\theta$.  We recover the travel time under the $i$-th ocean model by taking the path length divided by the $i$-th model velocity resulting in $\Delta t \frac{s_1(\theta,\bx,T_1)}{s_i(\theta,\bx,T_i)}$.  Recall, that we are trying to find the mean optimal time $u$ along this path so we must apply the probability distribution of realizing each velocity under the path resulting in the reachability time expression

    \begin{equation}
        u(\bx,\bx-\Delta t \cdot s_1(\theta,\bx,T_1)\be_\theta) = \Delta t \cdot \sum_{i=1}^{n} p(i) \frac{s_1 (\theta, \bx, T_1)}{s_i (\theta, \bx, T_i)}),
    \end{equation}
    which substituting into the above equation results in 
    \begin{equation}
        0 = \min\limits_{\substack{\theta \in [0,2\pi) \\ 0 {\color{violet} <}  s_i(\theta,\bx,T_i) \leq s_{i,max}(\theta, \bx,T_i) }} \{ - \nabla u \cdot \Delta t \cdot s_1(\theta,\bx,T_1)\be_\theta + \Delta t \cdot \sum_{i=1}^{n} p(i) \frac{s_1 (\theta, \bx, T_1)}{s_i (\theta, \bx, T_i)})\}.
    \end{equation}
    Observe that the right-hand side of the above expression is linear in $s_1(\theta,\bx,T_1)$ so it achieves its minimum at the endpoint or the largest possible value of $s_1(\theta,\bx,T_1)$.  Since we assume an optimal path exists, it follows that $s_1(\theta,\bx,T_1)=s_{1,max}(\theta,\bx,T_1)$ minimizes the right side of the equation.  Additionally, we note that the right-hand side of the above expression varies inversely with respect to each $s_i(\theta,\bx,T_i)$.  It follows that setting $s_i(\theta,\bx,T_i)=s_{i,max}(\theta,\bx,T_i)$ minimizes the right side of the equation and yields
    \begin{equation}
        0 = \min\limits_{\substack{\theta \in [0,2\pi) }} \{ - \nabla u \cdot \Delta t \cdot s_{1,max}(\theta,\bx,T_1)\be_\theta + \Delta t \cdot \sum_{i=1}^{n} p(i) \frac{s_{1,max} (\theta, \bx, T_1)}{s_{i,max} (\theta, \bx, T_i)})\}.
    \end{equation}
To complete the derivation, we divide by $\Delta t$ and obtain the governing PDE
 \eqref{eqn:unc-hj}:
    \begin{align}
        0 = \min\limits_{\substack{\theta \in [0,2\pi)}} \{ - \nabla u  \cdot s_{1,max}(\theta,\bx,T_1)\be_\theta + \sum_{i=1}^{n} p(i) \frac{s_{1,max} (\theta, \bx, T_1)}{s_{i,max} (\theta, \bx, T_i)}\}.
    \end{align}
Lastly, note that the above argument indicates that the optimal control direction $\theta^*$ corresponds to the Hamiltonian minimizer $\theta$ in \eqref{eqn:unc-hj}.  From the control directions $\theta^*$, the optimal path $\gamma^*$ from $\bx_S$ to an arbitrary target position $\bx$ is obtained.

In order to derive the other PDEs in \eqref{eqn:det-hj}, we take a similar approach. As before, we use the Dynamic Programming Principle to break the path, now $\gamma^*$, into two parts.  However, in this case we use the choice $dS = \Delta t \cdot s_{1,max}(\theta^*,\bx,T_1)\be_{\theta^*}$
made in the derivation above, to arrive at  
\begin{equation}
        T_i(\bx,\bx_S,\gamma^*) \approx T_i(\bx-\Delta t \cdot (s_{i,max}(\theta^*,\bx,T_i)\be_{\theta^*}),\bx_S,\gamma^*) + \Delta t.
\end{equation}
By applying a Taylor Expansion to the right-hand side, we get
\begin{equation}
        T_i(\bx,\bx_S,\gamma^*) = T_i(\bx,\bx_S,\gamma^*) - \Delta t \cdot \nabla T_i(\bx,\bx_S,\gamma^*) \cdot (s_{i,max}(\theta^*,\bx,T_i)\be_{\theta^*}) + \Delta t.
\end{equation}
Subtracting $T_i(\bx,\bx_S,\gamma^*)$ from each side results in
\begin{equation}
        0 = - \Delta t \cdot \nabla T_i(\bx,\bx_S,\gamma^*) \cdot (s_{i,max}(\theta^*,\bx,T_i)\be_{\theta^*}) + \Delta t.
\end{equation}
Again, division by $\Delta t$ results in
    \begin{equation}
        0 =- \nabla T_i(\bx,\bx_S,\gamma^*) \cdot (s_{i,max}(\theta^*,\bx,T_i)\be_{\theta^*}) + 1.
    \end{equation}
Finally, we can rewrite the above equation to yield \eqref{eqn:det-hj}:
    \begin{equation}
        1 = \nabla T_i(\bx,\bx_S,\gamma^*) \cdot (s_{i,max}(\theta^*,\bx,T_i)\be_{\theta^*}).
    \end{equation}
\end{proof}

\begin{remark}
    Notice that when $p(1)=1$ and $p(j)=0$ for $j=2,3,...,n$, this system of Hamilton-Jacobi PDEs simplifies to the deterministic reachability time PDE as presented in~\cite{Brandman2023}.
\end{remark}

\begin{remark}

    This derivation considers $dS=\Delta t \cdot s_1(\theta,\bx,T_1)\be_\theta$ to derive \eqref{eqn:unc-hj}.  While the derivation appears to prioritize $s_1(\theta,\bx,T_1)$ over the other ocean models, this turns out to not be the case due to $\sum_{i=1}^{n} p(i) \frac{s_1 (\theta, \bx, T_1)}{s_i (\theta, \bx, T_i)}$ acting as a corrective term in the PDE for the other ocean models. 
    In fact, we can derive a different PDE to place in the Hamilton-Jacobi PDE System.  For $j=1,2,...,n$, we could consider $dS=\Delta t \cdot v_j(\theta,\bx,T_j)\be_\theta$ yielding in the derivation the following PDE:
    \begin{equation*}
        \min_{\theta \in  [0, 2\pi)} \{ -\nabla u \cdot s_{j,max} (\theta, \bx, T_j) \be_\theta + \sum_{i=1}^{n} p(i) \frac{s_{j,max} (\theta, \bx, T_1)}{s_{i,max} (\theta, \bx, T_i)} \} = 0.
    \end{equation*}
    The resulting PDE systems, in numerical practice, seem to yield approximately equivalent solutions.  
\end{remark}

\subsection{Full Path Planning with Uncertainty Governing System}

Finally, we state the entire
governing system that we consider in this paper by summarizing \eqref{eqn:speed}, \eqref{eq:ODErewritten}, \eqref{eq:ODE_ic_rewritten}, \eqref{eqn:unc-hj}, \eqref{eqn:det-hj}, and \eqref{eqn:bc-hj}. With the given assumptions \eqref{eqn:ode_condition} and definitions \eqref{eqn:si_governing}-\eqref{eqn:theta_governing} derived in Section \ref{subsec:max} 
\begin{subequations}
\begin{align}
& s_{i,max}(\theta,\bx,T_i) = (\bv_{c,i}(\bx,T_i) \cdot \be_\theta) + \sqrt{s^2_{max} - (\bv_{c,i}(\bx,T_i) \cdot \be_\theta^\perp)^2} \label{eqn:si_governing}\\
& \be_\theta = [\cos(\theta),\sin(\theta)]^\top, \quad \theta^* = \text{Hamiltonian minimizer},
\label{eqn:theta_governing}
\end{align}
\end{subequations}
we aim to seek first the solution $u(\bx,\bx_S)$ with accompanying Hamiltonian minimizer $\theta^*$ and $T_i(\bx,\bx_S,\gamma^*)$ that satisfies the Hamilton Jacobi PDE system described in Section \ref{subsec:HJ} by solving \eqref{eqn:hj_unc_governing}-\eqref{eqn:hj_ic_governing}:  
\begin{subequations}\label{eqn:HJ}
\begin{align}
  & \min_{\theta \in  [0, 2\pi)} \{ -\nabla u \cdot s_{1,max} (\theta, \bx, T_1) \be_\theta + \sum_{i=1}^{n} p(i) \frac{s_{1,max} (\theta, \bx, T_1)}{s_{i,max} (\theta, \bx, T_i)} \} = 0 \label{eqn:hj_unc_governing}\\
    & \nabla T_i(\bx,\bx_S,\gamma^*) \cdot s_{i,max}(\theta^*,\bx, T_i)\be_{\theta^*} = 1 , \quad \text{for } i=1,...,n \label{eqn:hj_det_governing}\\
    & u(\bx_S,\bx_S) = 0 , \quad T_i(\bx_S,\bx_S,\gamma^*) = 0, \quad \text{for } i=1,...,n \label{eqn:hj_ic_governing}
\end{align}
\end{subequations}

The geometric path $\gamma^*$ connecting $\bx_S$ to a now fixed $\bx_E$ is completely determined by the optimal control $\theta^*(\bx,t)$ and the mean reachability time along $\gamma^*$ to $\bx_E$ is denoted $T^*$.  In order to find the parameterization $\by_i(t)$ of $\gamma^*$ according to which the vehicle would traverse $\gamma^*$ in the presence of the  $i$-th ensemble member $\bv_{c,i}$, we solve the following ODE backwards in time:

\begin{subequations}\label{eqn:ODE}

\begin{align}
    & \frac{d\by_i}{dt} = s_{i,max}(\theta^*,\by_i,t)\cdot \be_{\theta^*}(\by_i,t), \label{eqn:ode_governing}\\ 
    &\by_i(T_i^*)=\bx_E, \label{eqn:optimal_control} \\
    &T_i^*=T_i(\bx_E, \bx_S,\gamma^*).
    \label{eqn:optimal_control_2}
\end{align}

\end{subequations}

\section{Numerical Approximation}
\label{sec:3}

In this section, we present a numerical method to solve the Uncertainty Governing System \eqref{eqn:hj_unc_governing}-\eqref{eqn:hj_ic_governing} derived in the previous section. We start, in Section \ref{sec:LFFS_Intro}, with a discussion of the pre-existing Lax-Friedrich's Fast Sweeping~(LFFS) scheme for solving a single time-independent Hamilton-Jacobi PDE~\cite{Kao2004, Zhao2004}.  Given the necessity of solving multiple Hamilton-Jacobi PDEs, we aim in this section to novelly extend the LFFS scheme to solve Hamilton-Jacobi equations simultaneously.  Following this, Section \ref{sec:LFFS_ext} presents an initial generalization of fast sweeping to systems of Hamilton-Jacobi equations. 
For reasons that we explain below, this first attempt at generalization turns out to be inefficient; as a result, Section \ref{sec:LFFS_mod} proposes a more efficient extension of the LFFS scheme to systems of Hamilton-Jacobi PDEs.  Finally, we explain in Section \ref{sec:HJ_path} how to identify the optimal path $\gamma^*$ by incorporating the optimal control, determined by \eqref{eqn:hj_unc_governing}, into the ODE system \eqref{eqn:ode_governing}-\eqref{eqn:optimal_control_2}.

\subsection{Lax-Friedrich's Fast Sweeping}
\label{sec:LFFS_Intro}
Hamilton-Jacobi PDEs are difficult to solve numerically due to their nonlinearity, non-uniqueness, and potential non-differentiability~\cite{Crandall1983,Crandall1986}.  In particular, solutions of Hamilton-Jacobi equations are known to develop kinks: solutions of the eikonal equation quickly demonstrate this.  

Theoretically, unique solutions of Hamilton-Jacobi PDEs can be identified using the viscosity solution framework; this approach introduces an artificial viscosity (i.e. smoothing term) $\eta > 0$ into the regularized Hamiltonian~\cite{Evans2010}:
\begin{equation}
    \tilde{H}(\bx,\tilde{u},\nabla\tilde{u},\Delta\tilde{u}):=H(\bx,\tilde{u},\nabla\tilde{u}) - \eta \Delta \tilde{u} = 0.
\end{equation}
From this new Hamiltonian, we can define a unique viscosity solution as the solution $u$ obtained in the limit of vanishing viscosity:
\begin{equation}
    u(\bx,\bx_S) = \lim_{\eta \rightarrow 0} \tilde{u}(\bx,\bx_S).
\end{equation}
The vanishing viscosity framework is extremely valuable, as it provides a remarkable link between solutions of Hamilton-Jacobi PDE and value functions arising in optimal control \cite{Evans2010}.

In practice, the vanishing viscosity framework suggests that robust numerical methods for solving Hamilton-Jacobi PDE can be obtained through the solution of regularized problems in which sufficient numerical viscosity is added.  The challenge in designing such methods is to minimize the amount of diffusion introduced (to avoid smearing) while maintaining stability of the scheme.

In this paper, we focus on the formerly developed Lax-Friedrich's Fast Sweeping (LFFS) method~\cite{Kao2004, Zhao2004}.  This method is attractive since it can easily handle the nonlinear Hamiltonians present in path planning; in addition, it is computationally efficient (requiring only the setup and solve of a linear system at each iteration)~\cite{Chen2013}.  Furthermore, under a single Hamilton-Jacobi PDE, key properties of the viscosity solution attained from LFFS, including monotonic convergence, have been proven~\cite{Zhao2004}.

The starting point for LFFS is the Lax-Friedrichs discretization of a static regularized Hamilton-Jacobi PDE:

\begin{equation}
    H(\bx,u,\eta) - \eta \Delta u = R(\bx).
\end{equation}

This results in a nonlinear system of algebraic equations that must be solved in order to determine the solution of the PDE at each degree of freedom.  

For simplicity, the presentation here focuses on two spatial dimensions.  Let there be $m_1 \times m_2$ interior domain nodes.  Denote nodes $x_{j_1,j_2}$ where $j_1=0,1,...,m_1+1$ and $j_2=0,1,...,m_2+1$ with uniform cell size $\Delta x_1$ in the horizontal direction and $\Delta x_2$ in the vertical direction.  See Figure \ref{fig:2D Grid} for reference.
%%%%
\begin{figure}[!h]
\centering
\begin{tikzpicture}[scale=0.5]
  \def\nInt{8}    
  \def\mInt{6} 
  \pgfmathtruncatemacro{\nxone}{\nInt+1}
  \pgfmathtruncatemacro{\nyone}{\mInt+1}

  \draw[thick] (0,0) rectangle (\nxone,\nyone);

  \foreach \i in {0,...,\nxone}{
    \foreach \j in {0,...,\nyone}{
      \pgfmathtruncatemacro{\onB}{or(
         or(\i==0,\i==\nxone),
         or(\j==0,\j==\nyone)
      )}
      \ifnum\onB=1
        \fill[red] (\i,\j) circle (2.5pt);
      \else
        \fill[blue] (\i,\j) circle (2.5pt);
      \fi
    }
  }

  \coordinate (V1) at (3,3);
  \coordinate (V2) at (3,4);
  \draw (V1)--(V2);
  \node[right]  at (2.8,3.5)         {$\Delta x_2$};

  \coordinate (V3) at (2,3);
  \draw (V1)--(V3);
  \node[below]  at (2.55,3.1)         {$\Delta x_1$};

  \node[below left]  at (0,0)         {$x_{0,0}$};
  \node[below right] at (\nxone,0)     {$x_{m_1+1,0}$};
  \node[above left]  at (0,\nyone)     {$x_{0,m_2+1}$};
  \node[above right] at (\nxone,\nyone) {$x_{m_1+1,m_2+1}$};
\end{tikzpicture}
     \caption{An example of two dimensional grid setup}
    \label{fig:2D Grid}
\end{figure}
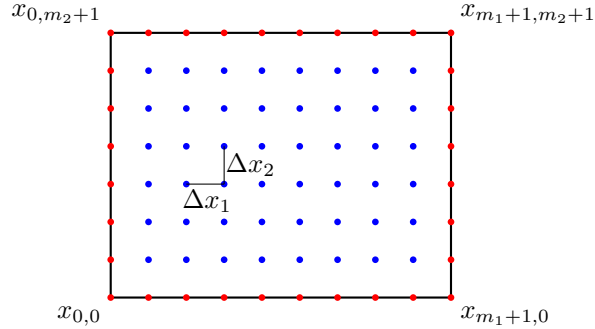
%%%%
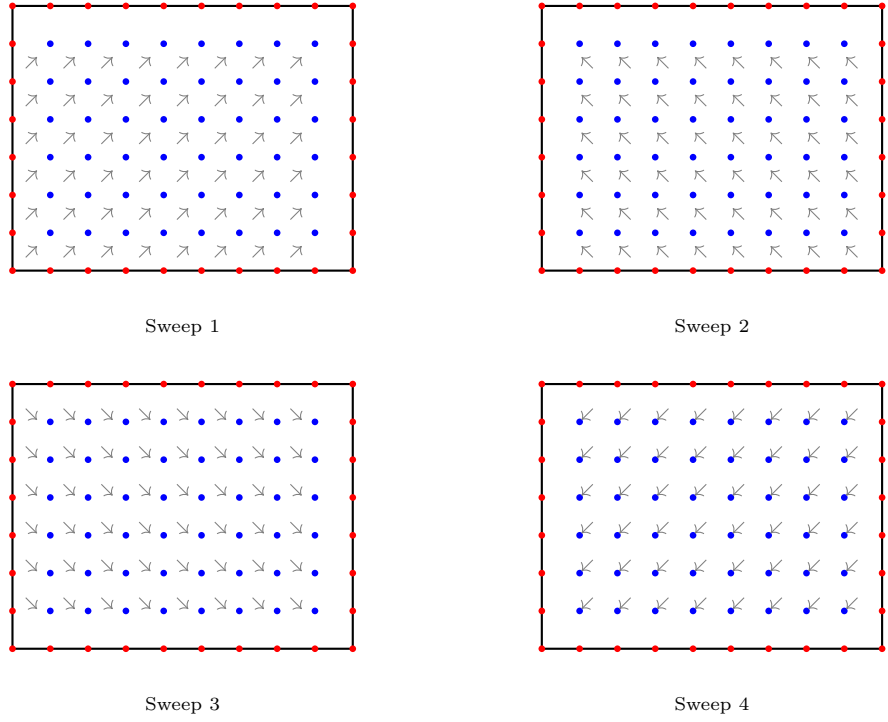
\begin{figure}[!h]
    \centering
\begin{tikzpicture}[scale=0.5]
  \def\nInt{8}
  \def\mInt{6}
  \pgfmathtruncatemacro{\nxone}{\nInt+1}
  \pgfmathtruncatemacro{\nyone}{\mInt+1}
  \def\dx{0.3}
  \def\dy{0.3}

  \draw[thick] (0,0) rectangle (\nxone,\nyone);
  \node at (4.5,-1.5) {\footnotesize Sweep 1};
  \foreach \i in {0,...,\nxone}{
    \foreach \j in {0,...,\nyone}{
      \pgfmathtruncatemacro{\onB}{or(or(\i==0,\i==\nxone), or(\j==0,\j==\nyone))}
      \ifnum\onB=1
        \fill[red] (\i,\j) circle (2.5pt);
      \else
        \fill[blue] (\i,\j) circle (2.5pt);
      \fi
    }
  }
  \foreach \i in {1,...,\nInt}{
    \foreach \j in {1,...,\mInt}{
      \draw[->,gray] (\i-1+0.35,\j-1+0.35) -- ++(\dx,\dy);
    }
  }

  \draw[thick,shift={(14,0)}] (0,0) rectangle (\nxone,\nyone);
  \node at (18.5,-1.5) {\footnotesize Sweep 2};
  \foreach \i in {0,...,\nxone}{
    \foreach \j in {0,...,\nyone}{
      \pgfmathtruncatemacro{\onB}{or(or(\i==0,\i==\nxone), or(\j==0,\j==\nyone))}
      \ifnum\onB=1
        \fill[red,shift={(14,0)}] (\i,\j) circle (2.5pt);
      \else
        \fill[blue,shift={(14,0)}] (\i,\j) circle (2.5pt);
      \fi
    }
  }
  \foreach \i in {1,...,\nInt}{
    \foreach \j in {1,...,\mInt}{
      \draw[->,gray] (14+\nxone-\i+0.35,\j-1+0.35) -- ++(-\dx,\dy);
    }
  }

  \draw[thick,shift={(0,-10)}] (0,0) rectangle (\nxone,\nyone);
  \node at (4.5,-11.5) {\footnotesize Sweep 3};
  \foreach \i in {0,...,\nxone}{
    \foreach \j in {0,...,\nyone}{
      \pgfmathtruncatemacro{\onB}{or(or(\i==0,\i==\nxone), or(\j==0,\j==\nyone))}
      \ifnum\onB=1
        \fill[red,shift={(0,-10)}] (\i,\j) circle (2.5pt);
      \else
        \fill[blue,shift={(0,-10)}] (\i,\j) circle (2.5pt);
      \fi
    }
  }
  \foreach \i in {1,...,\nInt}{
    \foreach \j in {1,...,\mInt}{
      \draw[->,gray] (\i-1+0.35,\nyone-\j-10+0.35) -- ++(\dx,-\dy);
    }
  }

  \draw[thick,shift={(14,-10)}] (0,0) rectangle (\nxone,\nyone);
  \node at (18.5,-11.5) {\footnotesize Sweep 4};
  \foreach \i in {0,...,\nxone}{
    \foreach \j in {0,...,\nyone}{
      \pgfmathtruncatemacro{\onB}{or(or(\i==0,\i==\nxone), or(\j==0,\j==\nyone))}
      \ifnum\onB=1
        \fill[red,shift={(14,-10)}] (\i,\j) circle (2.5pt);
      \else
        \fill[blue,shift={(14,-10)}] (\i,\j) circle (2.5pt);
      \fi
    }
  }
  \foreach \i in {1,...,\nInt}{
    \foreach \j in {1,...,\mInt}{
      \draw[->,gray] (14+\nxone-\i+0.35,\nyone-\j-10+0.35) -- ++(-\dx,-\dy);
    }
  }

\end{tikzpicture}
    \caption{A single set of sweeps on a given two dimensional  domain}
    \label{fig:Set_Of_Sweeps}
\end{figure}

Fast sweeping solves the nonlinear system arising from the Lax-Friedrichs discretization one unknown at a time.  This nonlinear system is nontrivial to solve, as it is generally nondifferentiable and potentially nonconvex.  The original work on fast sweeping focused on a Gauss-Seidel inspired scheme for solving the nonlinear system arising from an upwind discretization of a Hamilton-Jacobi PDE \cite{Rouy1992, Boue1999, Tsai2003}.  Here, we take a different approach and utilize the Lax-Friedrichs discretization \cite{Zhao2004, Kao2005}.  The resulting algorithm is simple and avoids the complicated upwinding calculations that would arise if we tried to solve our system using the original approach to fast sweeping.

Specifically, at each iteration, fast sweeping performs finite difference-based sweeps across the domain in all four directions (as shown in Figure \ref{fig:Set_Of_Sweeps}).  Updates to function values are only accepted if they result in values that are smaller than the previous iterate; this ensures that information propagates along characteristics in a manner consistent with the viscosity solution framework.  The method starts with a large constant initial guess.

We denote $\alpha=1,2,3, \text{ and } 4$ as the sweep index in a single set of (four) sweeps. In addition, $k$ is the LFFS iteration number corresponding to the number of sets of sweeps carried out.  We define $u^{k,\alpha}_j$ as the numerical approximation of $u$ at spatial coordinate $j=(j_1,j_2)$ on the corresponding $\alpha$ sweep in the sweep set for the $k$-th LFFS iteration.  Additionally, we denote $u^{(k)}$ as the matrix of all $u^{k,4}_j$ throughout the computational domain.  Note that the same discretization and notations for $u$ are utilized for $T_i$.

The Lax-Friedrich's discretization, as seen in \cite{Kao2004}, gives the following LFFS update equation for the $k$-th set of sweeps.
\begin{align}
    u_{j_1,j_2}^{k,\alpha} = c \cdot
    \left(R(x_{j_1,j_2}) - H \left(x_{j_1,j_2}, \frac{u_{j_1+1,j_2}^{k,\alpha-1} - u_{j_1-1,j_2}^{k,\alpha-1}}{2 \Delta x_1},\frac{u_{j_1,j_2+1}^{k,\alpha-1} - u_{j_1,j_2-1}^{k,\alpha-1}}{2 \Delta x_2} \right) \right) \nonumber \\ + c \cdot \left( \eta_{x_1}\frac{u_{j_1+1,j_2}^{k,\alpha-1} + u_{j_1-1,j_2}^{k,\alpha-1}}{2 \Delta x_1} +\eta_{x_2}\frac{u_{j_1,j_2+1}^{k,\alpha-1} + u_{j_1,j_2-1}^{k,\alpha-1}}{2 \Delta x_2} \right),
    \label{eqn:LFFS_interior}
\end{align}
where $c = 1 / \left({\frac{\eta_{x_1}}{\Delta x_1}+\frac{\eta_{x_2}}{\Delta x_2}}\right)$ and $\eta=(\eta_{x_1},\eta_{x_2})$.

In addition, we enforce the following boundary conditions in the two dimensional case after each sweeping direction, as again given in \cite{Kao2004}, to ensure boundary information flows outward:

\begin{align}
    &u_{0,j_2}^{k,\alpha} = \min(\max(2u_{1,j_2}^{k,\alpha} - u_{2,j_2}^{k,\alpha}, u_{2,j_2}^{k,\alpha}), u_{0,j_2}^{k,\alpha-1}), \nonumber \\
    &u_{m+1,j_2}^{k,\alpha} = \min(\max(2u_{m_1,j_2}^{k,\alpha} - u_{m_1-1,j_2}^{k,\alpha}, u_{m_1-1,j_2}^{k,\alpha}), u_{m_1+1,j_2}^{k,\alpha-1}), \nonumber \\
    &u_{j_1,0}^{k,\alpha} = \min(\max(2u_{j_1,1}^{k,\alpha} - u_{j_1,2}^{k,\alpha}, u_{j_1,2}^{k,\alpha}), u_{j_1,0}^{k,\alpha-1}), \nonumber \\
    &u_{j_1,m+1}^{k,\alpha} = \min(\max(2u_{j_1,m_2}^{k,\alpha} - u_{j_1,m_2-1}^{k,\alpha}, u_{j_1,m_2-1}^{k,\alpha}), u_{j_1,m_2+1}^{k,\alpha-1}).
    \label{eqn:LFFS_exterior}
\end{align}

 In a two dimensional setup, Algorithm \ref{alg:LFFS_Sweeps} illustrates how we carry out the following single set of sweeping directions under LFFS. Note that $u$ is being overwritten after each sweep in the sweep set.  This allows for storage efficiency.

 \begin{algorithm} [H]
\scriptsize
\caption{$k$ Sets of 2D LFFS Sweeps}\label{alg:LFFS_Sweeps}
\begin{algorithmic}

\State {For each iteration $k = 0,1, \cdots$}

\State Sweep 1: Solve $u_{j_1,j_2}^{k,1}$ using \eqref{eqn:LFFS_interior} for $j_1=1:m_1$, $j_2=1:m_2$
\State Assert boundary conditions \eqref{eqn:LFFS_exterior}
\State Sweep 2: Solve $u_{j_1,j_2}^{k,2}$ using \eqref{eqn:LFFS_interior} for $j_1=m_1:1$, $j_2=1:m_2$
\State Assert boundary conditions \eqref{eqn:LFFS_exterior}
\State Sweep 3: Solve $u_{j_1,j_2}^{k,3}$ using \eqref{eqn:LFFS_interior} for $j_1=1:m_1$, $j_2=m_2:1$
\State Assert boundary conditions \eqref{eqn:LFFS_exterior}
\State Sweep 4: Solve $u_{j_1,j_2}^{k,4}$ using \eqref{eqn:LFFS_interior} for $j_1=m_1:1$, $j_2=m_2:1$
\State Assert boundary conditions \eqref{eqn:LFFS_exterior}
\State $u^{k+1,0} = u^{k,4}$ %{\color{red}is this true? or  $u^{k} = u^{k,4}$ }

\end{algorithmic}
\end{algorithm}

For convergence, we iterate over a set of sweeps in all directions until an error threshold $\epsilon$ is met under the Frobenius norm $||\cdot||_F$:
$$
    || u^{(k+1)} - u^{(k)} ||_F \leq \epsilon.
$$
In this paper, we set $\epsilon= 10^{-4}$.

\subsection{Extension to Systems of Equations: Sweep Until Convergence Alternating Scheme}
\label{sec:LFFS_ext}

LFFS was previously used to compute the reachability time in the deterministic ocean model case~\cite{Brandman2023}.  However, in the present context we have a system of Hamilton-Jacobi PDEs to solve.  Adapting LFFS to such a system requires sweeping, in some manner, over the entire system of PDEs.  
Below, we present a straightforward, but inefficient, generalization of LFFS to systems of Hamilton-Jacobi equations.  Section \ref{sec:LFFS_mod} presents another approach that we demonstrate is more efficient.

A crude generalization of fast sweeping to systems is to first sweep one of the PDEs to convergence while keeping the others fixed.  Following this, we sweep a second PDE in the system to convergence while keeping the remaining variables fixed.  Iterating in this manner eventually results in each PDE being swept to convergence.

Once this is accomplished, we check if
$$
\max \{||u^{(1)}-u^{(0)}||_F,{||T_{i}^{(1)}-T_{i}^{(0)}||_F} \}<\epsilon.
$$
If this condition is met, we terminate the LFFS Alternations and have our solutions $u$ and $T_i$.  If the $\epsilon$-threshold is exceeded, we repeat the process described above.  This Sweep Until Convergence Alternating (SUCA-LFFS) approach is presented in Algorithm $\ref{alg:UntilConvSweepLFFS}$.

\begin{algorithm} [!h]
\scriptsize
\caption{Path Planning with Uncertainty - Sweep Until Convergence Alternating LFFS Scheme (SUCA-LFFS)}\label{alg:UntilConvSweepLFFS}
\begin{algorithmic}
\State 1. Initialize ${\bv_{c,i}}$, $s_{max}$, $\bx_S$, $u^{(0)}$ and ${T_i^{(0)}}$ where ${u(\bx_S,\bx_S)=T_i(\bx_S,\bx_S,\gamma)=0}$ on $\Omega$ for $i=1,...,n$
\State 2. Loop over the set of sweeps from LFFS:
\While{$\max \{||u^{(1)}-u^{(0)}||_F,{||T_{i}^{(1)}-T_{i}^{(0)}||_F} \}>\epsilon$}
\For{{each $b$ in $[u, T_i]$}}
\While {{$||b^{(k)}-b^{(k-1)}||_F >\epsilon$}} 
\For{each of the 4 sweeping directions}
\State LFFS on {\eqref{eqn:hj_unc_governing} with Newton's Method or \eqref{eqn:hj_det_governing} to update $u$ or $T_i$}
\State Enforce \eqref{eqn:LFFS_exterior}
\EndFor
\EndWhile
\State  Reset $k$ to 1
\EndFor
\EndWhile
\State 3. Pick $\bx_E$
\State 4. Forward Euler on \eqref{eqn:ode_governing} to obtain optimal path from $\bx_S$ to $\bx_E$
\end{algorithmic}
\end{algorithm}

\subsection{Modification of the Scheme: Single Sweep Set Alternating LFFS Scheme}

\label{sec:LFFS_mod}
The scheme presented in Section \ref{sec:LFFS_ext} is inefficient: sequentially solving each variable to convergence results in unnecessary overfitting.  In this section, we present an alternative approach.  

Rather than iteratively sweep each PDE to convergence, we propose that, at each sweep, the entire family of PDEs is updated.  We continue sweeping until we reach the termination condition

\begin{align*}
\max \{||u^{(k)}-u^{(k-1)}||_F,{||T_{i}^{(k)}-T_{i}^{(k-1)}||_F} \}<\epsilon.
\end{align*}

This Single Sweep Set Alternating (SSSA-LFFS) algorithm is presented in Algorithm \ref{alg:1SweepLFFS}.

This approach, in which the unknowns are swept simultaneously, is natural when fast sweeping is seen as a nonlinear analogue of Gauss-Seidel.  The resulting method avoids the overfitting mentioned previously; numerical results presented in Section \ref{sec:4} confirm that this yields significant efficiency gains.

\begin{algorithm} [!h]
\scriptsize
\caption{Path Planning with Uncertainty - Single Sweep Set Alternating LFFS Scheme (SSSA-LFFS)}\label{alg:1SweepLFFS}
\begin{algorithmic}
\State 1. Initialize ${\bv_{c,i}}$, $s_{max}$, $\bx_S$, $u^{(0)}$ and ${T_i^{(0)}}$ where ${u(\bx_S,\bx_S)=T_i(\bx_S,\bx_S,\gamma)=0}$ on $\Omega$ for $i=1,...,n$
\State 2. Loop over the set of sweeps from LFFS: 
\While{$\max \{||u^{(k)}-u^{(k-1)}||_F,{||T_{i}^{(k)}-T_{i}^{(k-1)}||_F} \}>\epsilon$}
\For{{each $b$ in $[u, T_i]$}}
\For{each of the 4 sweeping directions}
\State LFFS on {\eqref{eqn:hj_unc_governing} with Newton's Method or \eqref{eqn:hj_det_governing} to update $u$ or $T_i$}
\State Enforce \eqref{eqn:LFFS_exterior}
\EndFor
\EndFor

\EndWhile
\State 3. Pick $\bx_E$
\State 4. Forward Euler on \eqref{eqn:ode_governing} to obtain optimal path from $\bx_S$ to $\bx_E$
\end{algorithmic}
\end{algorithm}

\subsection{Hamiltonian Minimization Problem}
\label{sec:Ham_min}
In order to perform fast sweeping on the Hamilton-Jacobi system \eqref{eqn:hj_unc_governing}-\eqref{eqn:hj_ic_governing}, we need a robust and efficient means of computing the Hamiltonian present in \eqref{eqn:hj_unc_governing}.  We used Newton's method for this task due to its quadratic convergence for a suitable initial guess.

When only allowing $\theta$ to vary, under our admissible conditions, this Hamiltonian expression is a smooth nonlinear function that occurs over a $2 \pi$ period.  The function has a well-defined global minimum in all test cases included in this paper.  Due to the $2 \pi$ period and local convexity, we can split this problem into a small number of locally convex regions.  In practice, we noticed under our admissible conditions that this was no more than three locally convex regions.  Therefore, we consider three different initial guesses $\pi/2$, $\pi$, and $3\pi/2$, evaluate the Hamiltonian with them, and select the smaller of the three.  Under the Hamiltonian function structure, this guaranteed us in finding the true Hamiltonian minimizer.

To enhance the method's efficiency, we utilized the following strategy to reuse previously found optimal $\theta^*$.  For the first iteration of LFFS, we make an arbitrary guess as previously outlined.  However, for future iterations, we save the $\theta^*$ from the current iteration and use this value {\color{violet} in place of $\pi$} for the initial guess strategy for the next iteration, after the $T_i$ functions have been updated.

%\subsection{Constructing Optimal Path}
\subsection{Constructing the Optimal Path}
\label{sec:HJ_path}

The quantities $u, T_i,$ and $\theta^*$ suffice to solve the ODE \eqref{eqn:ode_governing}-   \eqref{eqn:optimal_control_2} for the optimal path $\gamma^*$.  In order to solve this ODE numerically, we employ Forward Euler backwards in time (since our ODE prescribes an end-time condition).

\subsection{Overall Algorithm }

Figure \ref{fig:overall_algorithm} illustrates the overall algorithm for the proposed method. As input, we provide the ocean‐current models \(
\bv_{c,i}\), the maximum speed \(s_{\max}\), and the starting location \(\bx_S\). Next, we initialize
$u^{(0)}$ and  $T_i^{(0)}$
as
$({2}/{s_{max}})\,\bigl\lvert \bx - \bx_S\bigr\rvert.$

We then apply our Single Sweep Set Alternating LFFS scheme to
\eqref{eqn:hj_unc_governing} and \eqref{eqn:hj_det_governing} as described in Section \ref{sec:LFFS_mod}. Once the convergence threshold is met, we store \(u\), \(T_i\), and \(\theta^*\). Finally, we solve
\eqref{eqn:ode_governing}
for the optimal path \(\by(t)\) over \([0,T^*]\), from \(\bx_S\) to \(\bx_E\), using a backward‐in‐time Forward Euler method.

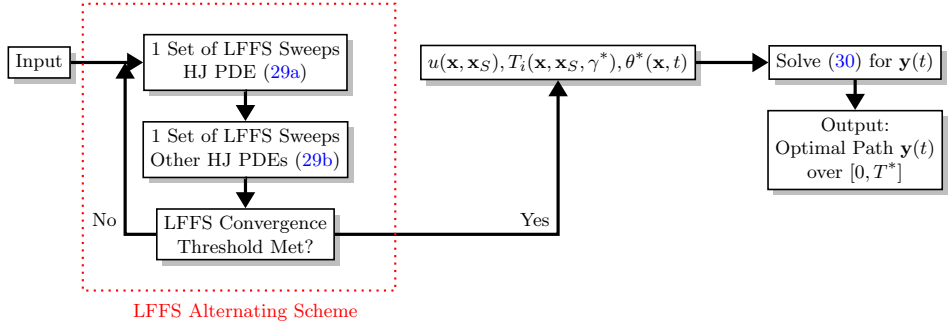
\begin{figure}[!h]
\centering
\begin{tikzpicture}[scale=0.8, transform shape,
    every node/.style={font=\small},
    node distance = 5mm and 12mm,
    arr/.style = {-Triangle,very thick},
    box/.style = {rectangle, draw, semithick,
                     minimum height=4mm, minimum width=7mm,
                     fill=white, drop shadow}
]

\node (n1) [box, align=center] {Input};

\node (n2) [box, align=center, right=of n1] {1 Set of LFFS Sweeps\\ HJ PDE \eqref{eqn:hj_unc_governing}};
\node (n3) [box, align=center, below=of n2] {1 Set of LFFS Sweeps\\ Other HJ PDEs \eqref{eqn:hj_det_governing}};
\node (n4) [box, align=center, below=of n3] {LFFS Convergence \\Threshold Met?};

\node (n5) [box, align=center, right=of n2] {$u(\bx,\bx_S),T_i(\bx,\bx_S,\gamma^*),\theta^*(\bx,t)$};

\node (n6) [box, align=center, right=of n5] {Solve \eqref{eqn:ODE} for $\by(t)$};
\node (n7) [box, align=center, below=of n6] {Output: \\Optimal Path $\by(t)$ \\ over $[0,T^*]$};

\node[align=center] 
  at ($([xshift=0mm, yshift=-45mm]n4.south)!0.5!(n2.south)$) 
  {\color{red} LFFS Alternating Scheme};

\draw[arr]   (n1) -- (n2.west);
\draw[arr]   (n2.south) -| (n3.north);
\draw[arr]   (n3.south) -| (n4.north);
\draw[arr]   (n4.west) -| ([xshift= -3mm] n2.west) node[midway,above left] {No} ;
\draw[arr]   (n4) -| (n5) node[midway, above left] {Yes};
\draw[arr]   (n5) -- (n6);
\draw[arr]   (n6) -- (n7);

\draw[red,thick,dotted] ([xshift= -10mm, yshift= 5mm] n2.north west)rectangle ([xshift= 10mm, yshift= -5mm] n4.south east);

\end{tikzpicture}
\caption{Illustration of the overall algorithm}
\label{fig:overall_algorithm}
\end{figure}

\section{Numerical Examples}
\label{sec:4}

This section presents several numerical examples that illustrate the validity and capabilities of the proposed path planning framework. Section \ref{sec:num_conv} describes an example in which the method converges to the expected semi-analytic solution.  Section \ref{num_eff} compares the two alternating fast sweeping methods proposed in Section \ref{sec:3}; in particular, we demonstrate the efficiency of the approach presented in Section \ref{sec:LFFS_mod}. The remaining examples in Sections \ref{path_comp} and \ref{five_mem_ens}  highlight the algorithm's performance in a variety of scenarios including more complicated and larger ensemble sizes.  

All examples are conducted on a uniform mesh where $\bx=(x_1,x_2)$ is in kilometers(km) and $t$ is in $10^3$ seconds(s). We set $s_{max}=1\text{ m/s}$.  The implementation is done in MATLAB\textsuperscript{\textregistered}\footnote{MATLAB is a registered trademark of The MathWorks, Inc.}, based on the framework depicted in Figure~\ref{fig:overall_algorithm}.

\subsection{Example 1: Convergence Test}
\label{sec:num_conv}

The goal of our first example is to demonstrate convergence of the Alternating LFFS scheme to a known semi-analytical solution.  To do this, we consider an ensemble of two ocean models ($n=2$) in which the ensemble members are identical, equi-probable, and scale linearly in time.

In this example, the computational domain is defined as $\Omega = [0\text{ km},10\text{ km}] \times [0\text{ km},10\text{ km}]$; the starting and target positions are given as $\bx_S = (5\text{ km},5\text{ km})$, and $\bx_E = (1\text{ km},0.5\text{ km})$, respectively.  The two ensemble members are identical and defined as
$$
\bv_{c,1}(\bx,t)=\bv_{c,2}(\bx,t)=[0, 0.002t]^T, \quad \bx\in \Omega, \quad t \in (0,\infty)
$$
where $p(1)=p(2)=0.5$.  We set the artificial viscosity $\eta=(\eta_{x_1},\eta_{x_2})=(1.5,1.5)$.  This may seem large, $O(1)$, but note that the order of magnitude of $u$ is 3 orders larger, $O(10^3)$.

There exists a semi-analytical solution for the reachability time $T^*$ under a linear-in-time current; $T^*$ satisfies the nonlinear equation
\begin{equation*}
    u(\bx_E,\bx_S) = T^* = \frac{l}{\sqrt{s_{max}^2 - \frac{\left(T^* \tilde{v}_c^{(2)}\right)^2}{4} + T^* \tilde{v}_c^{(2)} \hat{v}_c^{(2)} - \left(\hat{v}_c^{(2)}\right)^2 } + \frac{T^* \tilde{v}_c^{(1)}}{2} + \hat{v}_c^{(1)}} = t \tilde{\bv}_c + \hat{\bv}_c,
    \label{eq:analytic_sol}
\end{equation*}
where $\tilde{\bv}_c=[\tilde{v}_c^{(1)},\tilde{v}_c^{(2)}]^T \in \mathbb{R}^2$ and $\hat{\bv}_c=[\hat{v}_c^{(1)},\hat{v}_c^{(2)}]^T \in \mathbb{R}^2$.  A detailed derivation of this analytical solution is provided in Appendix \href{sec:app}{A}.

\begin{figure}[t]
  \begin{minipage}[t]{0.48\textwidth}
   \vspace{0pt}%
    \centering
    \includegraphics[trim={0cm 0cm 2cm 0cm},clip, width=\textwidth]{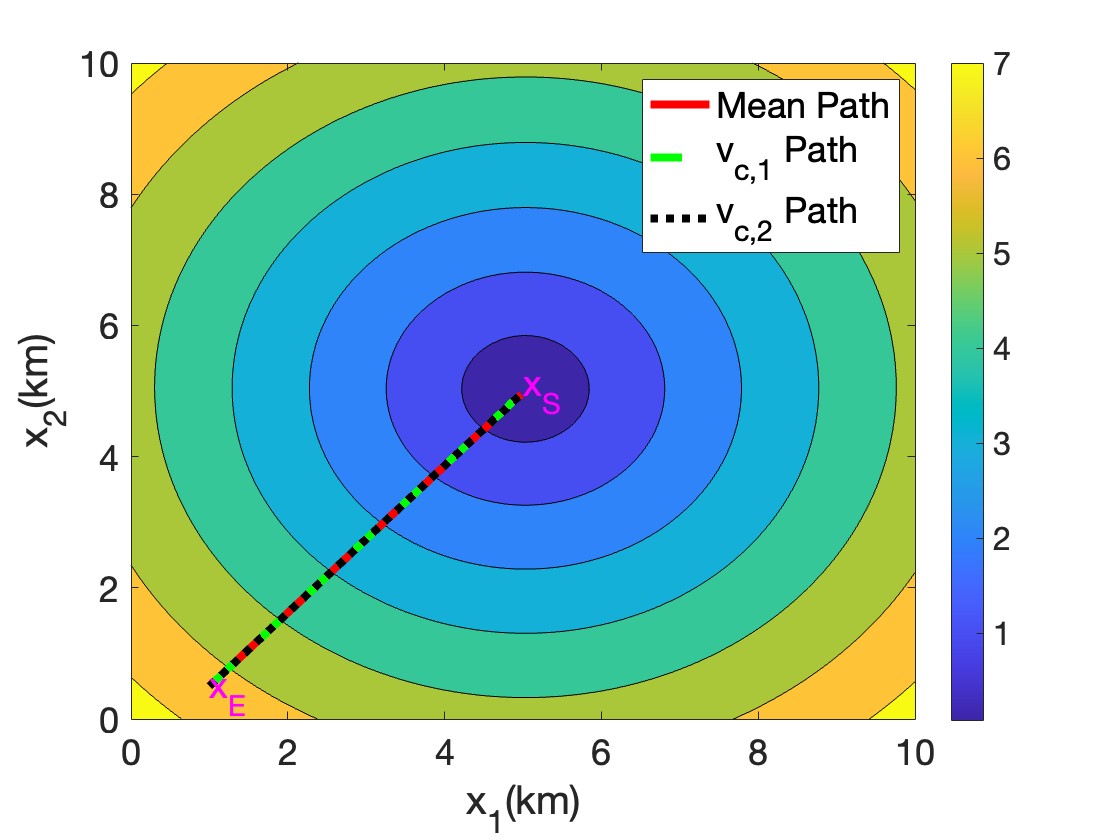}
  \end{minipage}
  \begin{minipage}[t]{0.48\textwidth}
    \centering
     \vspace{0pt}%
     \includegraphics[trim={0cm 0cm 2cm 0cm},clip, width=\textwidth]{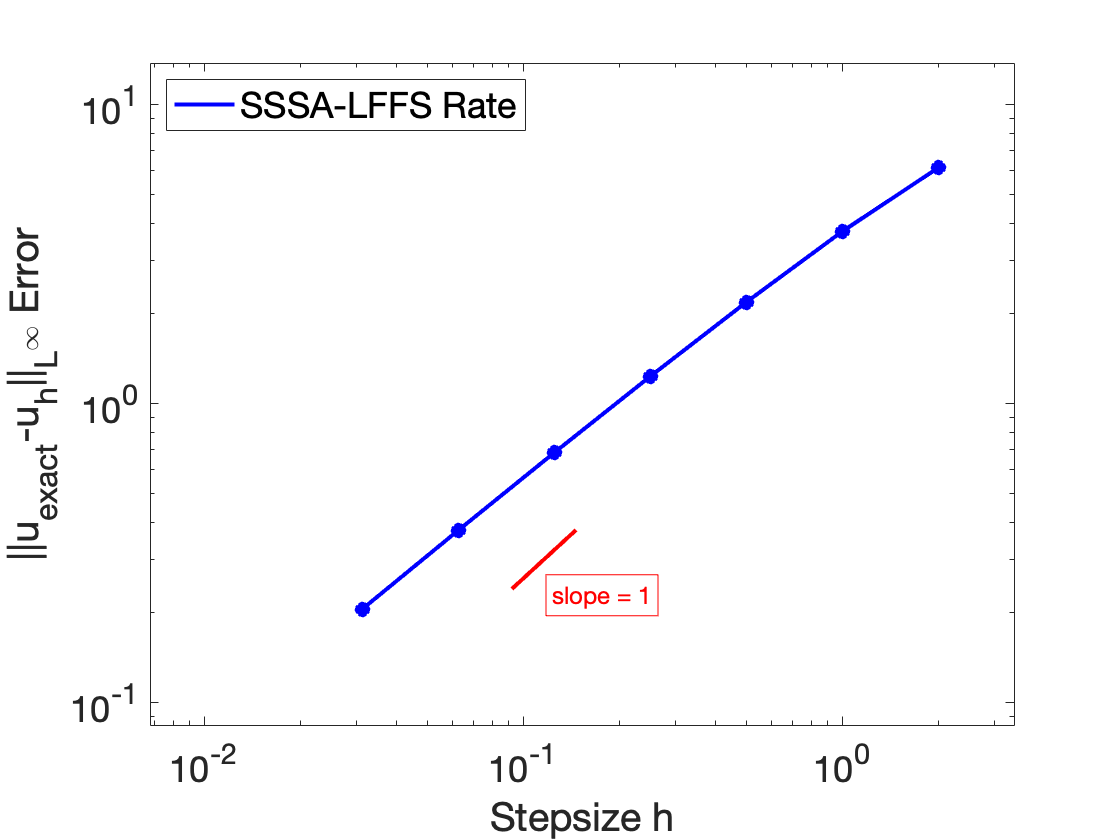}
  \end{minipage}
  \caption{
    Example 1: (left) Two deterministic paths (green and black) each with the same linear‑in‑time ocean current and the mean uncertainty path (red) from $\bx_S=(5\,$km$,5\,$km$)$ to $\bx_E=(1\,$km$,0.5\,$km$)$ plotted on the contour of $u$ in $10^3$ seconds for fixed mesh size; 
     (right) Illustration of the convergence rate for SSSA-LFFS. }
    \label{fig:lin_in_time}
\end{figure}

Our first step is to verify that, for a given mesh, the optimal path for the ensemble approach agrees with the path obtained using the deterministic framework~\cite{Brandman2023}.  To do this, we set the mesh spacing as $h=\Delta x_1 = \Delta x_2 = 0.125\text{ km}$ and ran the path planning algorithm.

The results, found in Figure \ref{fig:lin_in_time} (left), indicate agreement between the paths.  This validates that when the ensemble members are identical (zero uncertainty), then the mean uncertainty path coincides with the deterministic paths.  Specifically, it confirms that the uncertainty framework reduces to the deterministic framework if the ocean models $\bv_{c,i}$ are in perfect agreement.

We next conduct a convergence study by varying the node spacing $h$ and consequently the mesh size.  At each refinement step, we reduce $h$ by a factor of $2$ and apply the Single Sweep Set Alternating LFFS (SSSA-LFFS) Scheme to compute the approximate solution $u_h$ at the current mesh size.  For comparison, the exact solution $u_{\text{exact}}$ is obtained by solving for $T^*$ in \eqref{eq:analytic_sol} by Fixed Point Iteration for all points in the computational domain at the current mesh size.

Figure \ref{fig:lin_in_time} (right) shows that as the mesh is refined, there is a consistent decrease in error, indicating that $u_h$ is converging to $u_{\text{exact}}$.  Additionally, this result demonstrates a near-linear convergence rate in $L^\infty$ norm of the error for the Single Sweep Set Alternating LFFS scheme.  This is expected as the original LFFS scheme optimally has a linear convergence rate~\cite{Kao2004} and is known to degrade based on the advection term in the PDE~\cite{Fomel2009}.

\subsection{Example 2: Comparison of proposed fast sweeping methods: computational efficiency}
\label{num_eff}

The goal of this example is to compare, quantitatively, the computational efficiency of the Single Sweep Set Alternating (SSSA-LFFS) Scheme and Single Sweep Until Convergence Alternating (SSUCA-LFFS) Scheme, as detailed in Algorithm~\ref{alg:UntilConvSweepLFFS} and Algorithm~\ref{alg:1SweepLFFS}, respectively.  To do this, we consider the performance of each method on various two-member ocean current ensembles.  
The computational domain is defined as $\Omega = [-100\text{ km},100\text{ km}] \times [-100\text{ km},100\text{ km}]$, where the UUV starting position is given as $\bx_S = (0\text{ km},0\text{ km})$.

Below, we describe two different setups of ensemble members.  Setup 1 is designed with ensemble members that exhibit less nonlinearity, in general, in comparison to the possible members under Setup 2.  In both cases, we choose a 2-member ensemble from that setup, assuming  $p(1)=p(2)=0.5$, and discretize space using $161 \times 161$ grid points.  We set $\eta=(\eta_{x_1},\eta_{x_2})=(1.75,1.75)$.  This increase in the artificial viscosity magnitude is required as Setup 2 exhibits greater nonlinearity and more instability.

\begin{itemize}
    \item \textbf{Setup 1:} We test four different two-member ensemble configurations formed from two of the following time-dependent currents:
$$
\bv_a = [0, 0.0002t]^T, \quad 
\bv_b = [0, -0.0002t]^T, \quad 
\bv_c = [0.4\sin(tx_1), 0.4\cos(tx_2)]^T.
$$

    \item \textbf{Setup 2:}  
    We construct ensemble members as vortex ocean current models as described in~\cite{Brandman2023}.  This class of currents take the form: 
\begin{equation*}
\bv_{\text{vortex}}(\bx,t)=\frac{1}{\zeta(10^{-5}||\bw||_2+\frac{4}{3})||\bw||_2}*\alpha
\begin{bmatrix}
                    -w_{2} \\
                    w_{1}
\end{bmatrix},  \ 
\bw = \begin{bmatrix}
                    x_{1} - x_{1_S}- \beta t \\
                    x_{2} - x_{2_S}
                \end{bmatrix},
\label{eq:vortex}
\end{equation*}
    where $\alpha \in \{-1,1\}, \beta \in \mathbb{R},$ 
    and $\zeta \in \mathbb{R}_{\ne 0}$ are model parameters dictating the spinning direction, moving velocity, and strength of the vortex, respectively.  We test four different two-member ensemble configurations formed from two of the following vortex currents:

    \begin{itemize}
        \item $\bv_d$: counterclockwise right-moving vortex $(\alpha=1,\beta=1,\zeta=1)$
        \vspace{0.2cm}
        \item $\bv_e$: clockwise right-moving vortex $(\alpha=-1,\beta=1,\zeta=1)$
        \vspace{0.2cm}
        \item $\bv_f$: counterclockwise left-moving vortex $(\alpha=1,\beta=-1,\zeta=1)$
        \vspace{0.2cm}
        \item $\bv_g$: weaker counterclockwise right-moving vortex $(\alpha=1,\beta=1,\zeta=2)$
        \vspace{0.2cm}
    \end{itemize}
\end{itemize}

The results for both setups, summarized in Tables \ref{tab:2a} and \ref{tab:2b}, demonstrate that the Single Sweep Set Alternating method outperforms the Sweep Until Convergence Alternating method across all test cases.  The improved efficiency of the Single Sweep Set method is particularly pronounced in Setup 2, presumably due to the nonlinearity introduced by the presence of vortices~\cite{Miksis2022}.  Therefore, the Single Sweep Set Alternating Scheme appears as the more computationally efficient method, especially in the presence of greater nonlinearity.  

It is also worth noting that the total number of sweeps under the Sweep Until Convergence Alternating Scheme results in a non-uniform number of sweeps to update $u$, $T_1$, and $T_2$ while Single Sweep Set Alternating enforces a enforces a uniform amount of sweeps.  This is important to note as a single sweep to update $u$ takes a significant amount of time, due to the added minimization problem to solve, compared with a sweep to update any $T_i$.  This means that the total sweep sets of $u$ and $T_i$ combined required for system solution convergence under the Sweep Until Convergence Alternating Scheme does not linearly scale with time; whereas, Single Sweep Set Alternating, due to the uniformity, still does.  This lack of scaling is demonstrated when comparing the first and third test scenarios under Set 2.

\begin{table}[h!]
    \centering
                \begin{tabular}{|c|c|c|c|c|c|c|} \hline 

                \multicolumn{2}{|c|}{Ensemble} &
                \multicolumn{2}{|p{3.7cm}|}{\centering Sweep Until Convergence Alternating Method}
                &
                \multicolumn{2}{|p{3.7cm}|}{\centering Single Sweep Set Alternating Method}&
                \multicolumn{1}{|c|}{Speedup Factor}\\ \hline

                $\bv_{c,1}$ & $\bv_{c,2}$ & Total Sweep Sets & Time($s$) & Total Sweep Sets & Time($s$) &  \\ \hline

                $\bv_a$ & $\bv_a$ & 448 & 93.25 & 267 & 46.83 & 1.991\\ \hline

                $\bv_a$ & $\bv_b$ & 390 & 84.74 & 267 & 45.96 & 1.844\\ \hline

                $\bv_c$ & $\bv_c$ & 2651 & 362.43 & 441 & 99.51 & 3.642\\ \hline

                $\bv_a$ & $\bv_c$ & 1438 & 172.55 & 315 & 58.74 & 2.951\\ \hline

            \end{tabular}
    \caption{Example 2: Total number of sweep sets and computational time for 2 LFFS Alternating Schemes given 2 non-vortex ensemble members scenarios selected from Setup 1.}
    \label{tab:2a}
\end{table}

\begin{table}[h!]
    \centering
                \begin{tabular}{|c|c|c|c|c|c|c|} \hline 

                \multicolumn{2}{|c|}{Ensemble} &
                \multicolumn{2}{|p{3.7cm}|}{\centering Sweep Until Convergence Alternating Method}
                &
                \multicolumn{2}{|p{3.7cm}|}{\centering Single Sweep Set Alternating Method}&
                \multicolumn{1}{|p{2.2cm}|}{\centering Speedup Factor (Time)}\\ \hline

                $\bv_{c,1}$ & $\bv_{c,2}$ & Total Sweep Sets & Time($s$) & Total Sweep Sets & Time($s$) & \\ \hline

                $\bv_d$ & $\bv_d$ & 9562 & 2157.56 & 1473 & 322.01 & 6.700\\ \hline

                $\bv_d$ & $\bv_e$ & 4544 & 1119.49 & 1485 & 315.09 & 3.553\\ \hline

                $\bv_d$ & $\bv_f$ & 10463 & 1938.04 & 1854 & 391.18 & 4.954\\ \hline
                %10463

                $\bv_d$ & $\bv_g$ & 5692 & 1681.85 & 1380 & 301.69 & 5.575\\ \hline

            \end{tabular}
    \caption{Example 2: Total number of sweep sets and computational time for 2 LFFS Alternating Schemes given 2 moving vortex ensemble members scenarios.  Note that the solution convergence for the ensemble of $\bv_d$ and $\bv_f$ under Sweep Until Convergence Alternating Method takes more sweeps but less time than that of the ensemble  of $\bv_d$ and $\bv_d$ due the non-uniformity of the number of sweeps for $u$, $T_1$ and $T_2$ under the method and the $\bv_d$ and $\bv_d$ requiring more sweeps for $u$, which these are more computationally expensive compared with $T_i$ due to the minimization problem.}
    \label{tab:2b}
\end{table}

We further consider the ensemble of $\bv_a$ and $\bv_b$ from Setup 1 as we wish to show how the maximum error between successive sweep sets of $u$, $T_1$, and $T_2$ behave under each method.  Figure \ref{fig:SUCAvsSSSA} shows that sweeping until convergence produces sharp declines in the maximum error for one of the solutions in the system.  However, as mentioned previously this goes against the Gauss-Seidel sweeping nature of the algorithm resulting in an overcorrection as seen with much more pronounced spikes when we then try to solve for one of the other solutions in the system.  The more cautious SSSA-LFFS still sees this overcorrection, but the impact is significantly lower and resulted in the method requiring less sweep steps for all solutions to meet our required stopping criteria.

\begin{figure}[t!]
    \centering
    \includegraphics[scale = 0.2]{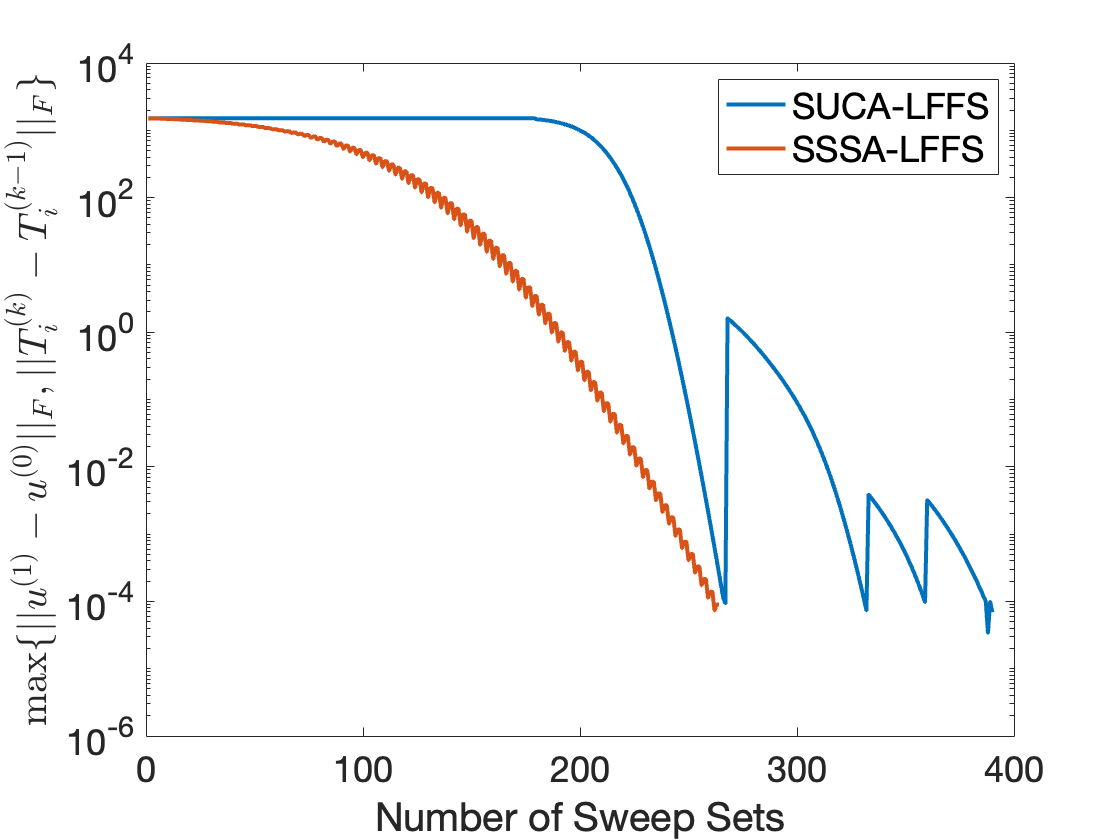}
    \caption{Example 2: The maximum error between successive sweep sets of $u$, $T_1$, and $T_2$ under Sweep Until Convergence Alternating Lax-Friedrich's Fast Sweeping (SUCA-LFFS) and Single Sweep Set Alternating Lax-Friedrich's Fast Sweeping (SSSA-LFFS) for the ensemble of $\bv_a$ and $\bv_b$ from Setup 1 until stopping criteria is achieved.}
    \label{fig:SUCAvsSSSA}
\end{figure}

\subsection{Example 3. Understanding the relationship between the optimal ensemble path and the optimal deterministic paths}
\label{path_comp}
The goal of this section is to understand, through two examples, the relationship between the path $\gamma^*$ that is time-optimal for the ensemble and the deterministic paths that are time-optimal for each of the ensemble members viewed in isolation.  

Our starting point is a two-member ensemble for which the individual reachability times, $T_1$ and $T_2$ as obtained from solving \eqref{eqn:hj_unc_governing}-\eqref{eqn:hj_ic_governing}, do not differ appreciably: $\frac{T_1}{T_2}\approx 1$.  
In this case, we show that the optimal mean-time path $\gamma^*$  is, roughly speaking, an interpolant of the optimal paths obtained by $T_1$ and $T_2$.

In contrast, we then consider a two-member ensemble for which $T_1$ and $T_2$ differ significantly: $\frac{T_1}{T_2} >> 1$ in at least part of the domain.
As a result, the optimal mean-time path $\gamma^*$ is not a straightforward interpolant of the optimal deterministic paths associated with the individual ensemble members.

\subsubsection{Example 3a:  $\frac{T_1}{T_2} \approx 1$}

For our first example, we choose ensemble members for which $T_1 \approx T_2$. The computational domain is defined as $\Omega = [0\text{ km},10\text{ km}] \times [0\text{ km},10\text{ km}]$; the starting and target positions are $\bx_S = (5\text{ km},5\text{ km})$ and $\bx_E = (1\text{ km},0.5\text{ km})$, respectively.  The numerical experiment is conducted on a $161 \times 161$ nodal mesh.  

The ocean current model ensemble members are defined as 
$$
\bv_{c,1}(\bx,t)=[0.002t]^T, \quad 
\bv_{c,2}(\bx,t)=[0.4\sin(tx_1), 0.4\cos(tx_2)]^T,
$$
where $p(1)=p(2)=0.5$.  We set $\eta=(\eta_{x_1},\eta_{x_2})=(1.5,1.5)$.

The optimal mean reachability time $u(\bx)$ and the deterministic reachability times $T_1(\bx)$ and $T_2(\bx)$ are visualized in Figure \ref{fig:3a_contours}.  The contours of $u$, exhibiting similarities with those of both $T_1$ and $T_2$, appear to be an interpolation between the two. Furthermore, Figure \ref{fig:low_vort} indicates that the optimal mean-time path $\gamma^*$ (red) lies between the two deterministic paths formed from $T_1$ (purple) and $T_2$ (black).

\begin{figure}[H]
     \centering
     \begin{subfigure}[b]{0.3\textwidth}
         \centering
         \includegraphics[width=\textwidth]{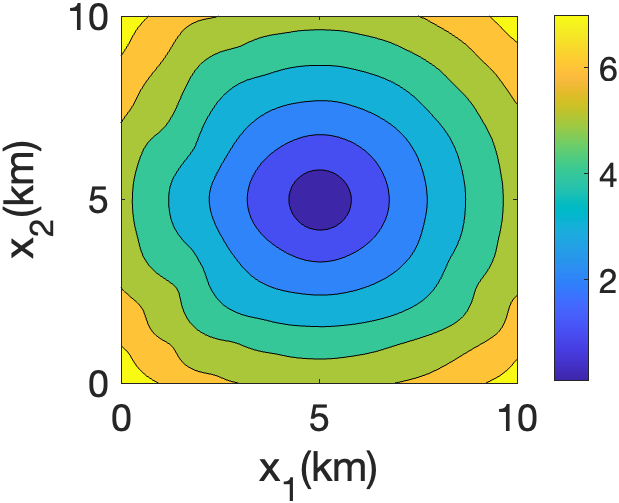}
         \caption{$u$}
         \label{fig:3a_u}
     \end{subfigure}
     \hfill
     \begin{subfigure}[b]{0.3\textwidth}
         \centering
         \includegraphics[width=\textwidth]{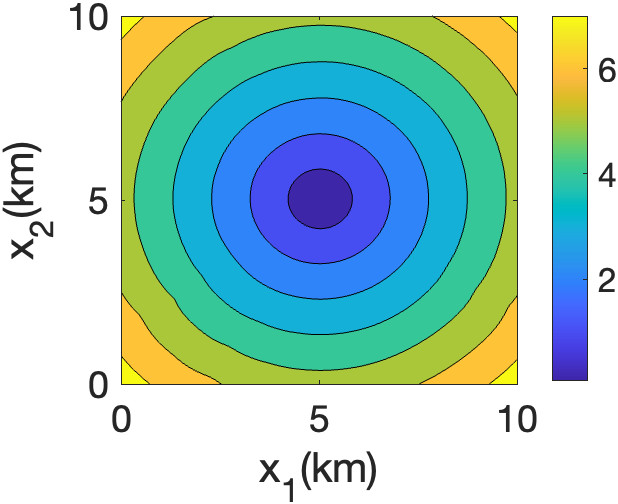}
         \caption{$T_1$}
         \label{fig:3a_T1}
     \end{subfigure}
     \hfill
     \begin{subfigure}[b]{0.3\textwidth}
         \centering
         \includegraphics[width=\textwidth]{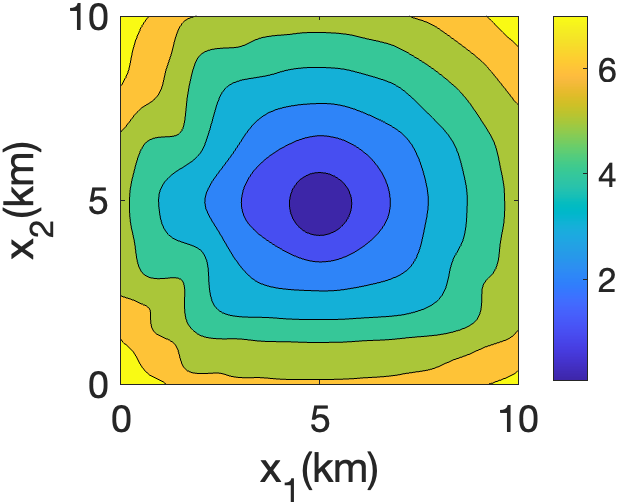}
         \caption{$T_2$}
         \label{fig:3a_T2}
     \end{subfigure}
        \caption{Example 3a: Contour plots of the reachability times $u$, $T_1$, and $T_2$ in $10^3$ seconds.}
        \label{fig:3a_contours}
\end{figure}

\begin{figure}[H]
    \centering
    \includegraphics[scale = 0.2]{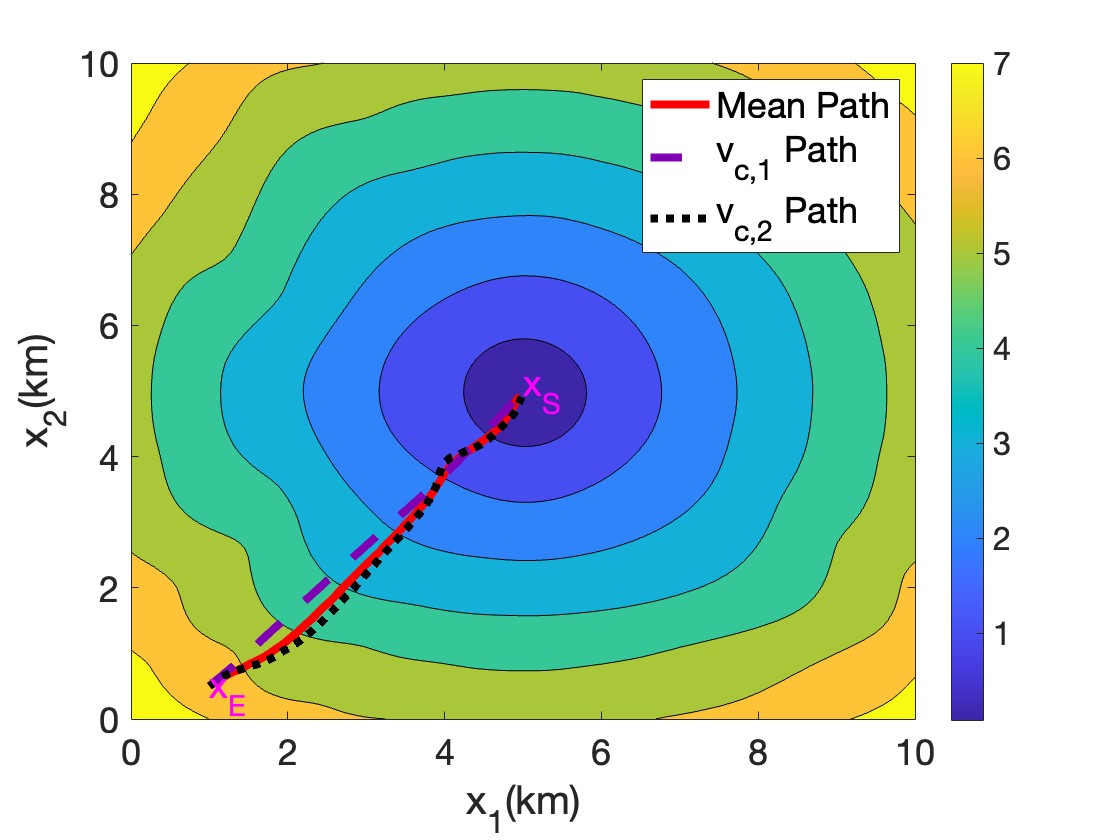}
    \caption{Example 3a: Two deterministic paths (purple and black) and the mean uncertainty path (red) from $(5\text{ km},5\text{ km})$ to $(1\text{ km},0.5\text{ km})$ on top of contour of $u$ in $10^3$ seconds.}
    \label{fig:low_vort}
\end{figure}

\begin{figure}[t!]
     \centering
     \begin{subfigure}[b]{0.45\textwidth}
         \centering
         \includegraphics[width=\textwidth]{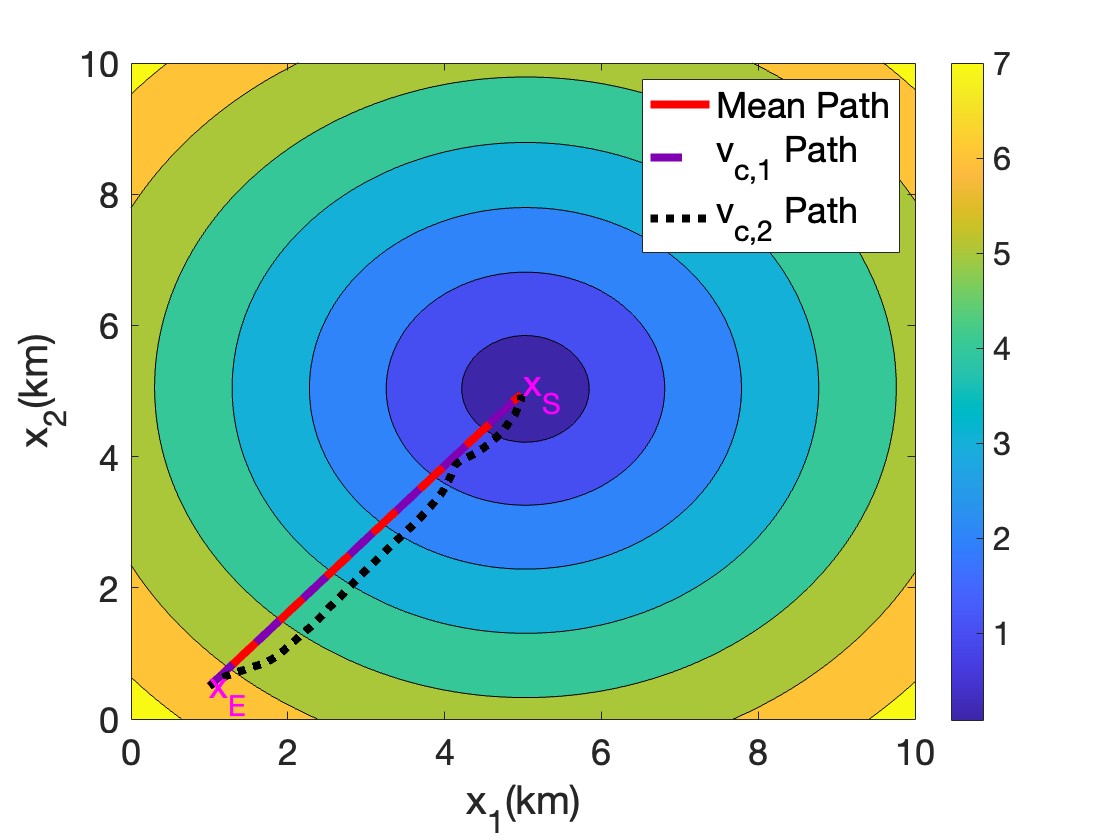}
         \caption{$p(1)=1, \quad p(2)=0$}
         \label{fig:3a_10}
     \end{subfigure}
     %\hfill
     \begin{subfigure}[b]{0.45\textwidth}
         \centering
         \includegraphics[width=\textwidth]{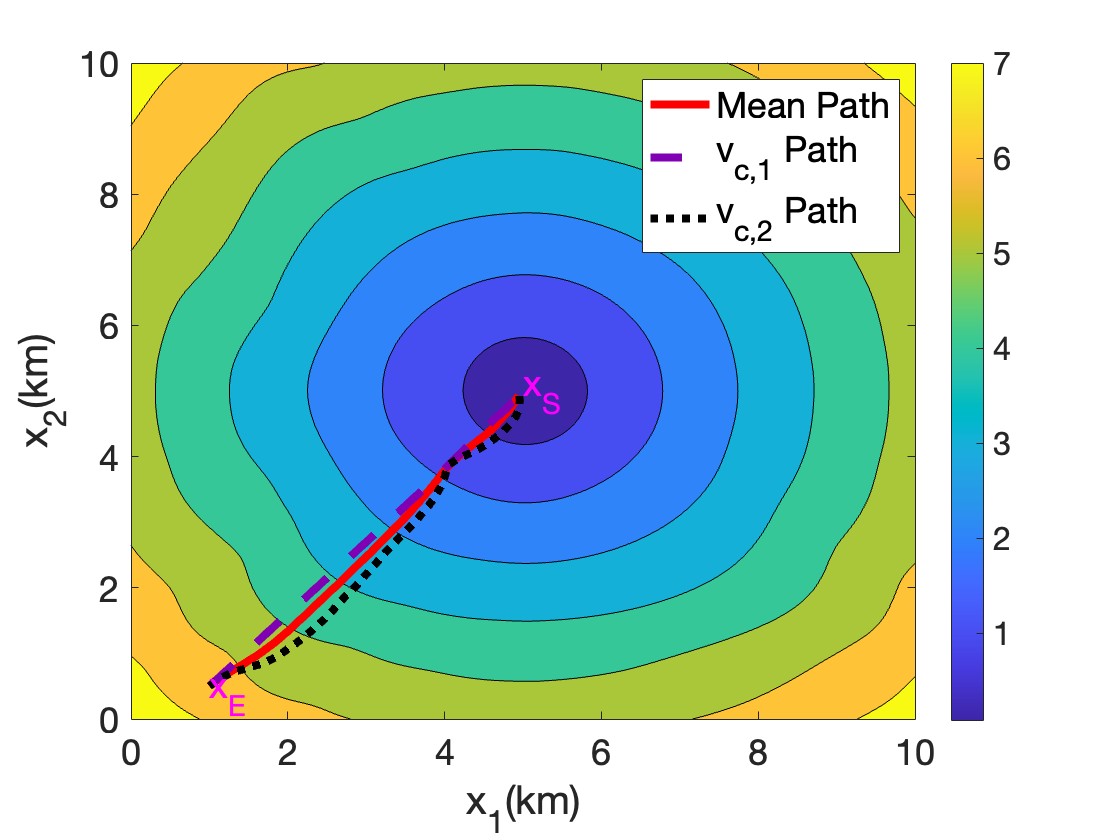}
         \caption{$p(1)=0.7, \quad p(2)=0.3$}
         \label{fig:3a_73}
     \end{subfigure}
     %\hfill
     \begin{subfigure}[b]{0.45\textwidth}
         \centering
         \includegraphics[width=\textwidth]{Figures/3afig55.jpg}
         \caption{$p(1)=0.5, \quad p(2)=0.5$}
         \label{fig:3a_55}
     \end{subfigure}
     %\hfill
     \begin{subfigure}[b]{0.45\textwidth}
         \centering
         \includegraphics[width=\textwidth]{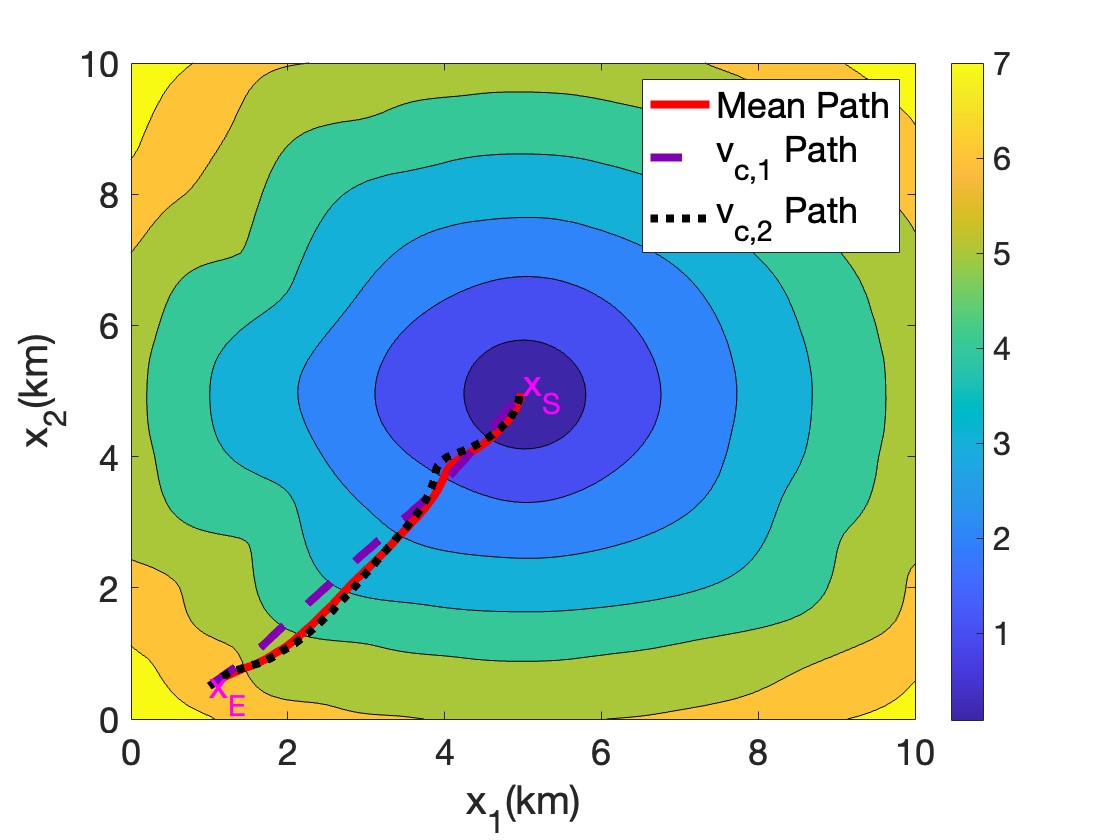}
         \caption{$p(1)=0.3, \quad p(2)=0.7$}
         \label{fig:3a_37}
     \end{subfigure}
     \begin{subfigure}[b]{0.45\textwidth}
         \centering
         \includegraphics[width=\textwidth]{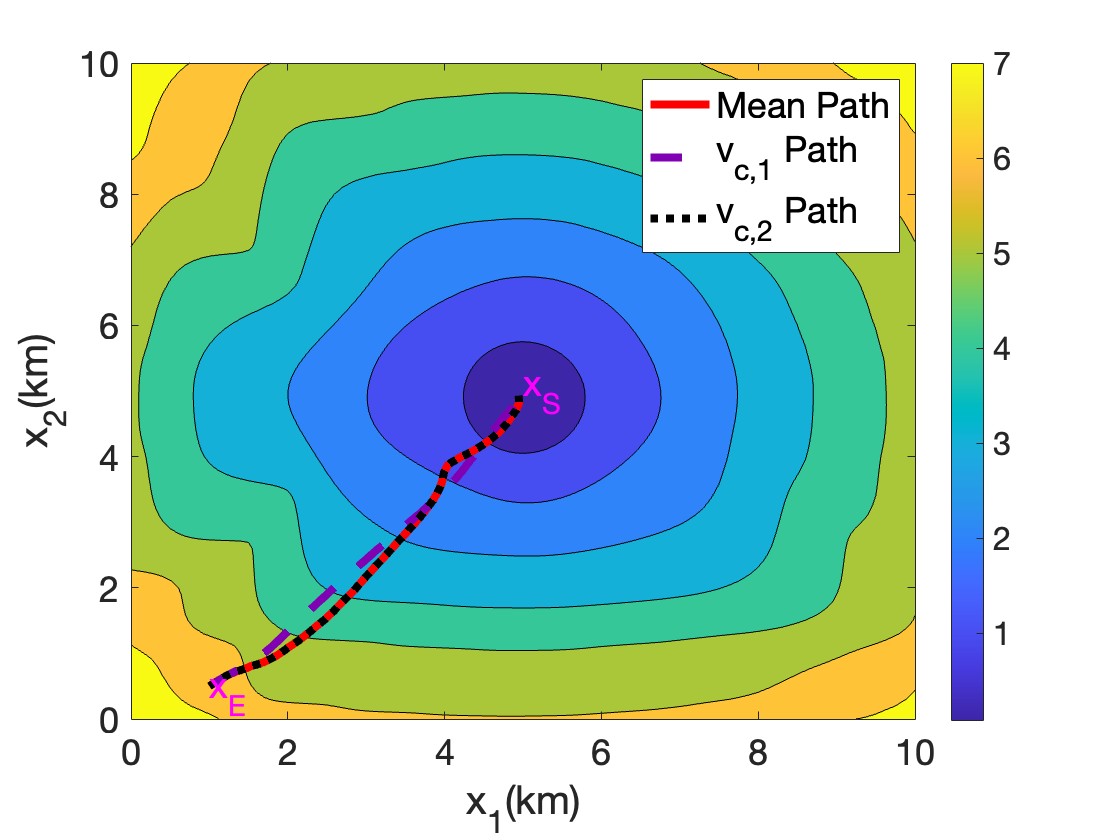}
         \caption{$p(1)=0, \quad p(2)=1$}
         \label{fig:3a_01}
     \end{subfigure}
\caption{ Example 3a: The mean uncertainty path under different probabilities.}
\label{fig:3a_prob}
\end{figure}

Figure \ref{fig:3a_prob} tests how different probabilities of realization affect the mean reachability time contour plot $u$ and the mean uncertainty path derived.  We note that as $p(1) \rightarrow 1$ and $p(2) \rightarrow 0$, $u$ begins to closely resemble $T_1$ and the optimal path constructed from $u$ converges to the path constructed from $T_1$.  The same holds with the reversal of the probabilities.

\subsubsection{Example 3b:  $\frac{T_1}{T_2} >> 1$}

\begin{figure}[t]
     \centering
     \begin{subfigure}[b]{0.3\textwidth}
         \centering
         \includegraphics[width=\textwidth]{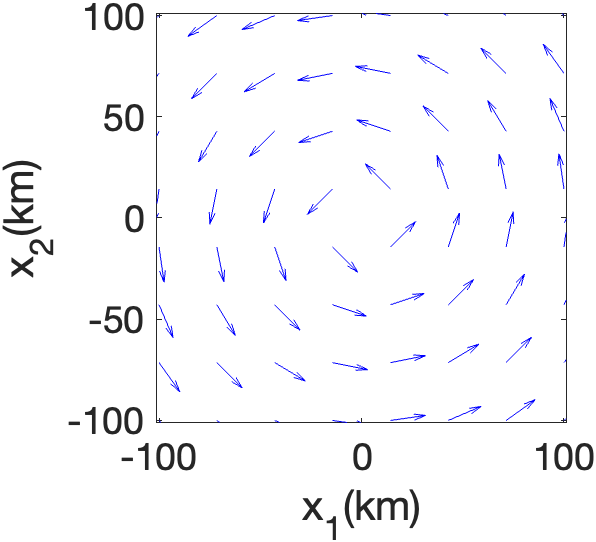}
         \caption{$\bv_{c,1}$ at $t=0\text{ s}$}
         \label{fig:3b_v_c1_t0}
     \end{subfigure}
     \hfill
     \begin{subfigure}[b]{0.3\textwidth}
         \centering
         \includegraphics[width=\textwidth]{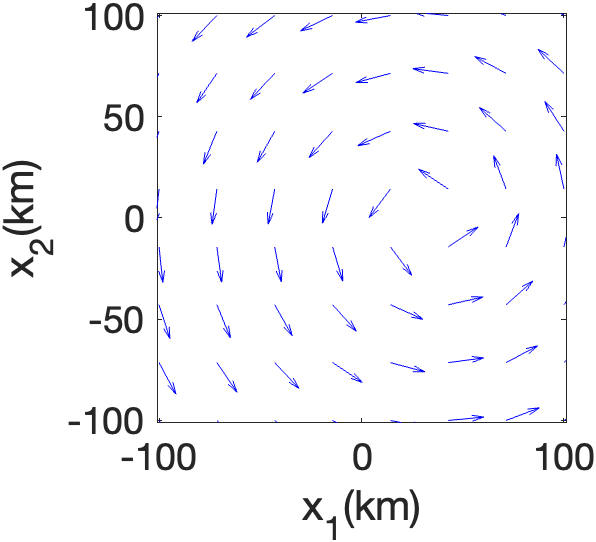}
         \caption{$\bv_{c,1}$ at $t=33.3 *10^3\text{ s}$}
         \label{fig:3b_v_c1_t1}
     \end{subfigure}
     \hfill
     \begin{subfigure}[b]{0.3\textwidth}
         \centering
         \includegraphics[width=\textwidth]{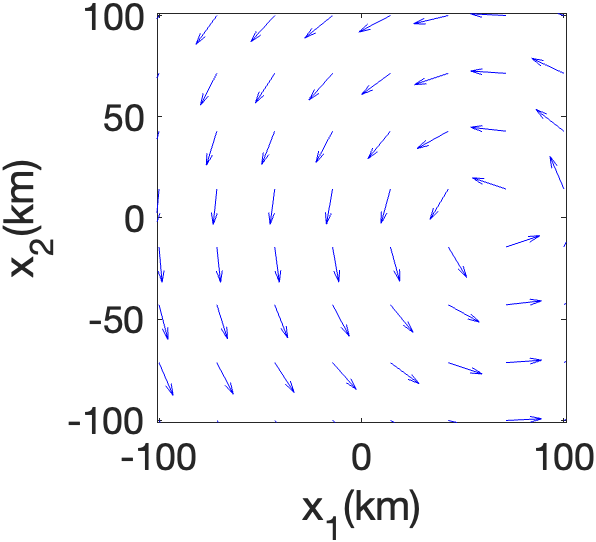}
         \caption{$\bv_{c,1}$ at $t=66.7 *10^3\text{ s}$}
         \label{fig:3b_v_c1_t2}
     \end{subfigure}
     \begin{subfigure}[b]{0.3\textwidth}
         \centering
         \includegraphics[width=\textwidth]{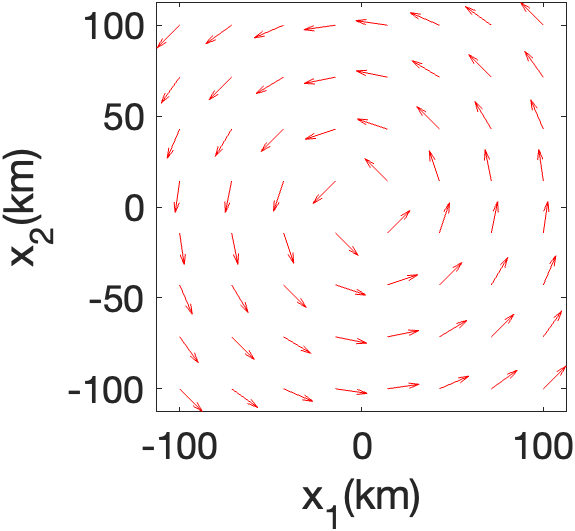}
         \caption{$\bv_{c,2}$ at $t=0\text{ s}$}
         \label{fig:3b_v_c2_t0}
     \end{subfigure}
     \hfill
     \begin{subfigure}[b]{0.3\textwidth}
         \centering
         \includegraphics[width=\textwidth]{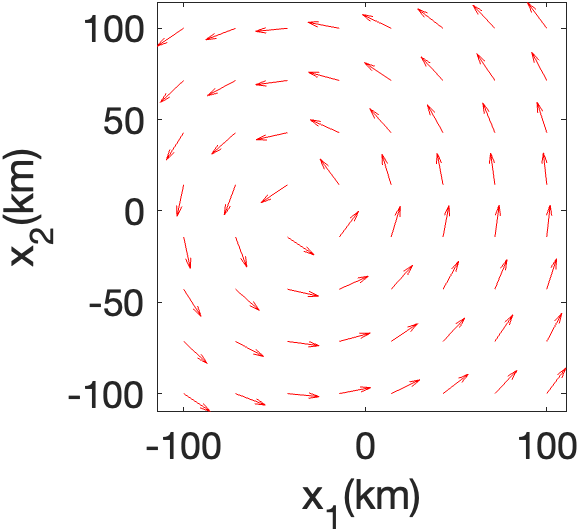}
         \caption{$\bv_{c,2}$ at $t=33.3 *10^3\text{ s}$}
         \label{fig:3b_v_c2_t1}
     \end{subfigure}
     \hfill
     \begin{subfigure}[b]{0.3\textwidth}
         \centering
         \includegraphics[width=\textwidth]{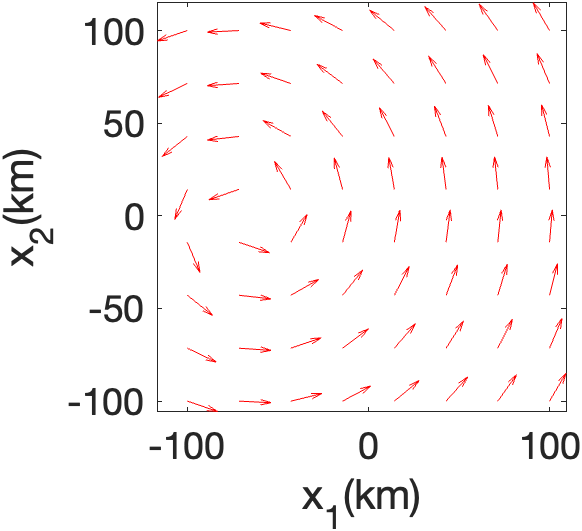}
         \caption{$\bv_{c,2}$ at $t=66.7 *10^3\text{ s}$}
         \label{fig:3b_v_c2_t2}
     \end{subfigure}
    \caption{Example 3b: Quiver plot snapshots in time of a left moving dimensionalized counterclockwise vortex ocean model and a right moving model over the same spatial domain $\Omega$. }
    \label{fig:high_vort_quivers}
\end{figure}

For our next example, we choose ensemble members for which $T_1$ and $T_2$ differ significantly.  The computational domain is defined as $\Omega = [-100\text{ km},100\text{ km}] \times [-100\text{ km},100\text{ km}]$; the starting and target positions are $\bx_S = (0\text{ km},0\text{ km})$ and $\bx_E = (50\text{ km},50\text{ km})$, respectively.  The numerical example is conducted on a $61 \times 61$ nodal mesh.
The ocean current model ensemble members are defined as 
\begin{align}
\bv_{c,1}(\bx,t) &=\frac{1}{(10^{-5}||\bw||_2+\frac{4}{3})||\bw||_2}*\begin{bmatrix}
            -w_2 \\
            w_1
        \end{bmatrix}, \text{ where } \bw = \begin{bmatrix}
            x_1 - t \\
            x_2
        \end{bmatrix}, \\
\bv_{c,2}(\bx,t)&=\frac{1}{(10^{-5}||\tilde{\bw}||_2+\frac{4}{3})||\tilde{\bw}||_2}*\begin{bmatrix}
            -\tilde{w}_2 \\
            \tilde{w}_1
        \end{bmatrix}, \text{ where } \tilde{\bw} =\begin{bmatrix}
            x_1 + t \\
            x_2
\end{bmatrix},
\end{align}
where $p(1)=p(2)=0.5$.  We set $\eta=(\eta_{x_1},\eta_{x_2})=(1.75,1.75)$.

These ensemble members define moving vortices; $\bv_{c,1}$ is moving to the right and $\bv_{c,2}$ is moving to the left.  Fixed-time snapshots of these ocean current models are provided in Figure \ref{fig:high_vort_quivers}, highlighting the movement of each vortex.

The contours for the optimal mean reachability time $u$, and the deterministic reachability times $T_1$ and $T_2$ are plotted in Figure \ref{fig:3b_contours}.  The significant differences between the contours $T_1$ and $T_2$ reflect the divergent dynamics between the two ensemble members.  As a result of these differences, the contours for $u$ do not resemble the contours of either ensemble member's reachability time.

Because the contours of $u$ do not resemble those of either $T_1$ and $T_2$, we do not expect the optimal mean-time trajectory to resemble either of the optimal deterministic trajectories.  This is confirmed by Figure \ref{fig:high_vort}, which plots all three paths.

\begin{figure}
     \centering
     \begin{subfigure}[b]{0.3\textwidth}
         \centering
         \includegraphics[width=\textwidth]{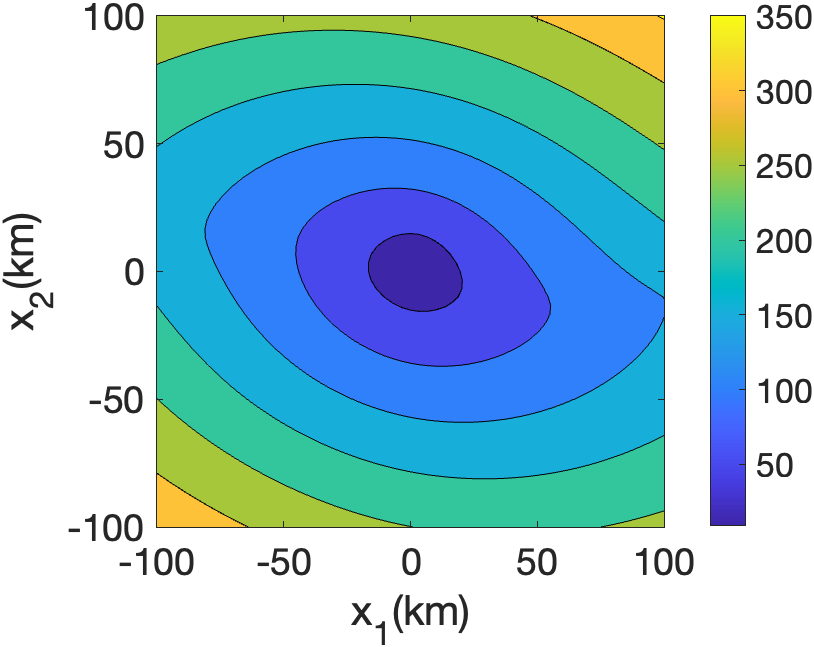}
         \caption{$u$}
         \label{fig:3b_u}
     \end{subfigure}
     \hfill
     \begin{subfigure}[b]{0.3\textwidth}
         \centering
         \includegraphics[width=\textwidth]{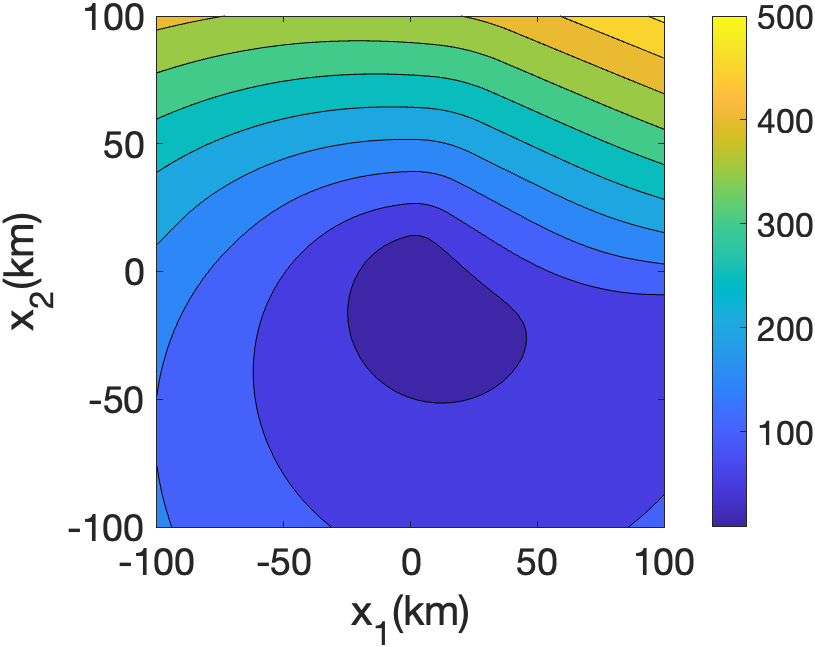}
         \caption{$T_1$}
         \label{fig:3b_T1}
     \end{subfigure}
     \hfill
     \begin{subfigure}[b]{0.3\textwidth}
         \centering
         \includegraphics[width=\textwidth]{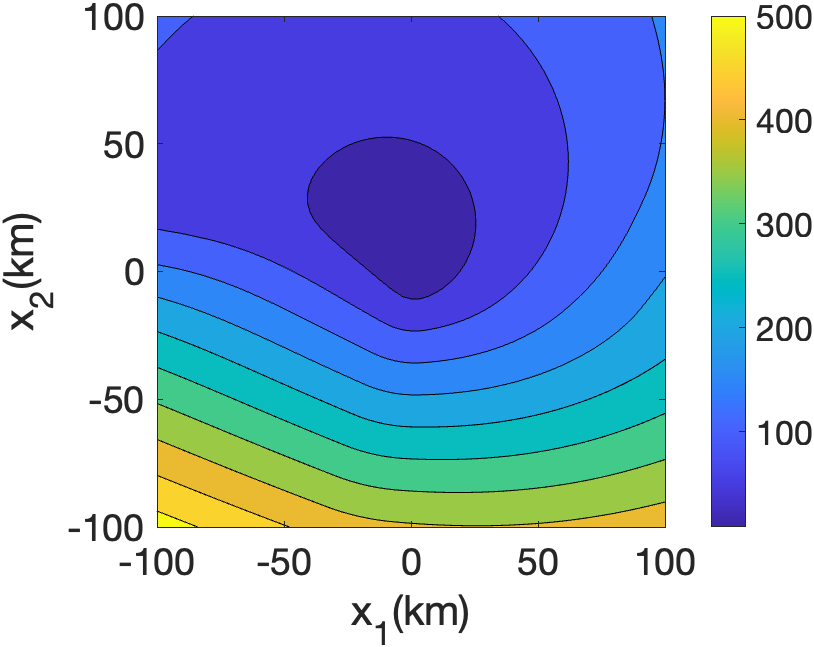}
         \caption{$T_2$}
         \label{fig:3b_T2}
     \end{subfigure}
        \caption{Example 3b: The contour plots for the solutions $u$, $T_1$, and $T_2$ in $10^3$ seconds  given a left moving dimensionalized counterclockwise vortex ocean model or a right moving model.}
        \label{fig:3b_contours}
\end{figure}

\begin{figure}
    \centering
    \includegraphics[scale = 0.2]{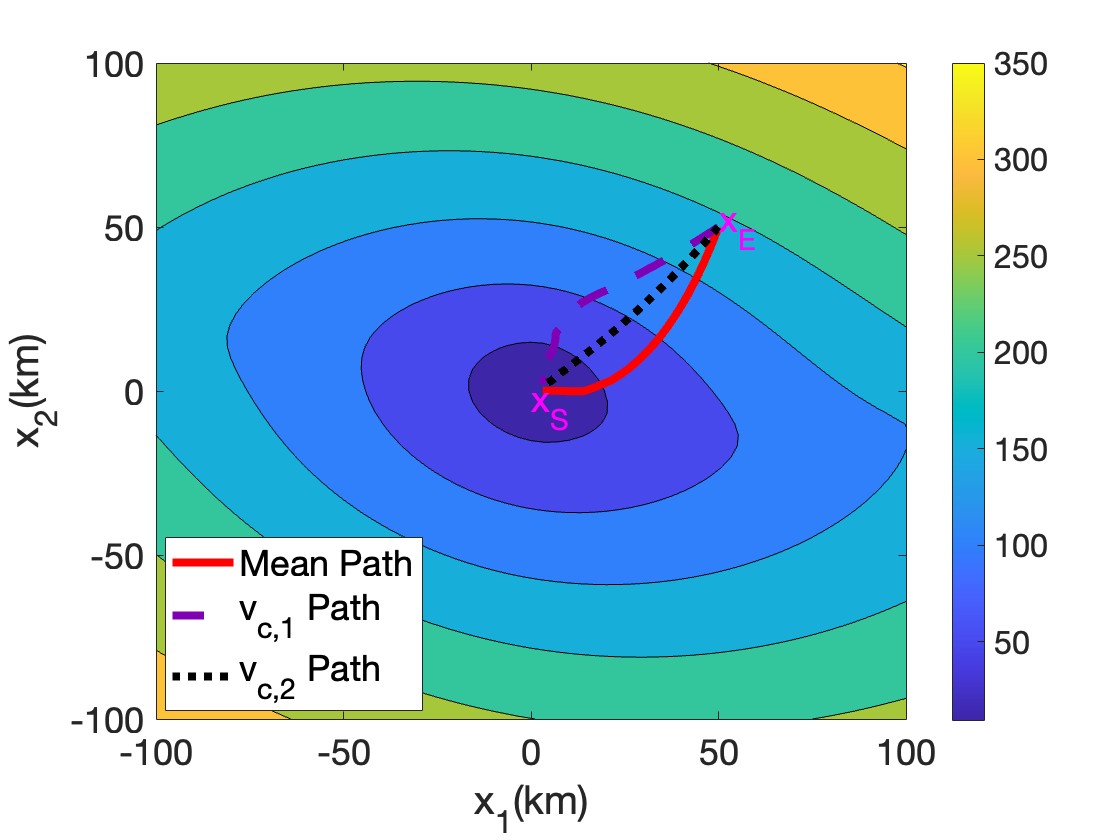}
    \caption{Example 3b: Two deterministic paths (purple and black) each with a different left or right moving dimensionalized counterclockwise vortex and the mean uncertainty path (red) from $(0\text{ km},0\text{ km})$ to $(50\text{ km},50\text{ km})$ with $p(1)=p(2)=0.5$ on top of the contour plot of $u$ in $10^3$ seconds.}
    \label{fig:high_vort}
\end{figure}

\subsection{Example 4. Five-Member Ensembles}
\label{five_mem_ens}

In this set of examples, we demonstrate the ability to consider ensembles of more than two ocean current models.  Specifically, we focus on examples with five distinct ensemble members. 
As in several of the previous examples, the computational domain is defined as $\Omega = [0\text{ km},10\text{ km}] \times [0\text{ km},10\text{ km}]$; the starting and target positions are $\bx_S = (5\text{ km},5\text{ km})$ and $\bx_E = (1\text{ km},0.5\text{ km})$, respectively.  We set $\eta=(\eta_{x_1},\eta_{x_2})=(1.5,1.5)$.

\subsubsection{Example 4a: Ensemble members are constant vectors exhibiting symmetry}

\begin{figure}[t!]
     \centering
     \begin{subfigure}[b]{0.3\textwidth}
         \centering
         \includegraphics[width=\textwidth]{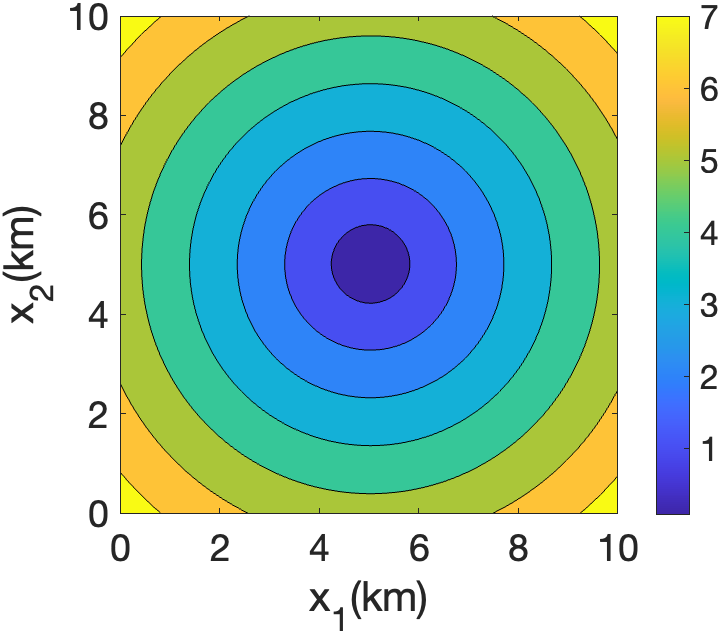}
         \caption{$u$}
         \label{fig:4_u}
     \end{subfigure}
     \hfill
     \begin{subfigure}[b]{0.3\textwidth}
         \centering
         \includegraphics[width=\textwidth]{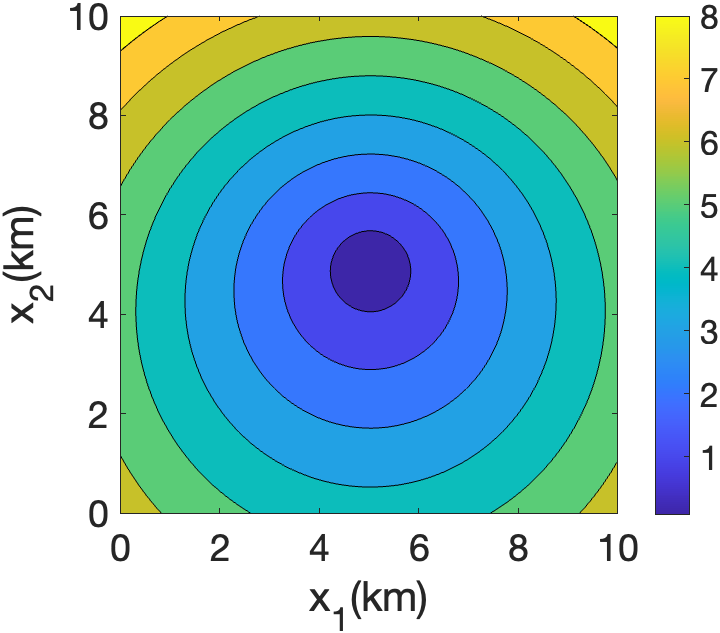}
         \caption{$T_1$}
         \label{fig:4_T1}
     \end{subfigure}
     \hfill
     \begin{subfigure}[b]{0.3\textwidth}
         \centering
         \includegraphics[width=\textwidth]{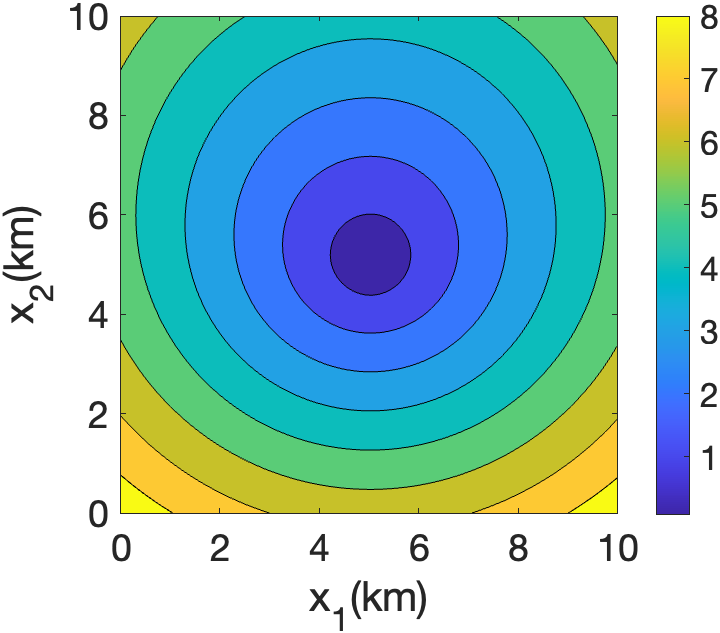}
         \caption{$T_2$}
         \label{fig:4_T2}
     \end{subfigure}
     \begin{subfigure}[b]{0.3\textwidth}
         \centering
         \includegraphics[width=\textwidth]{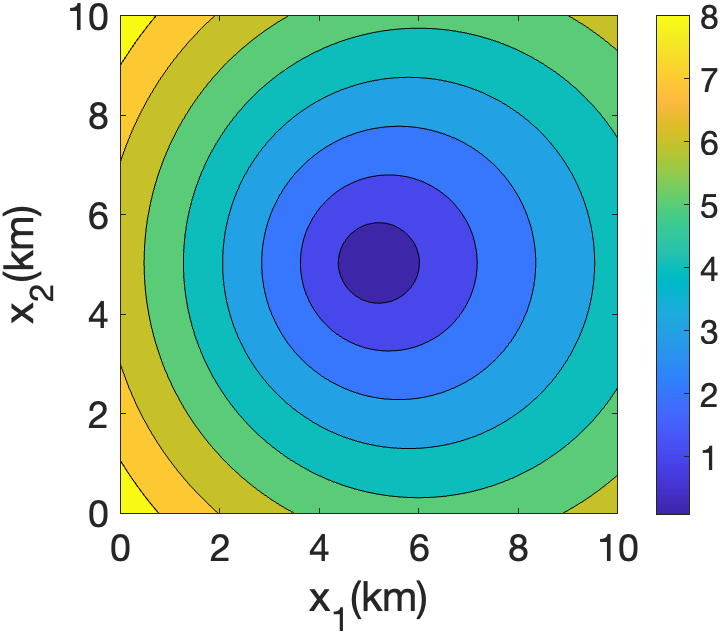}
         \caption{$T_3$}
         \label{fig:4_T3}
     \end{subfigure}
     \hfill
     \begin{subfigure}[b]{0.3\textwidth}
         \centering
         \includegraphics[width=\textwidth]{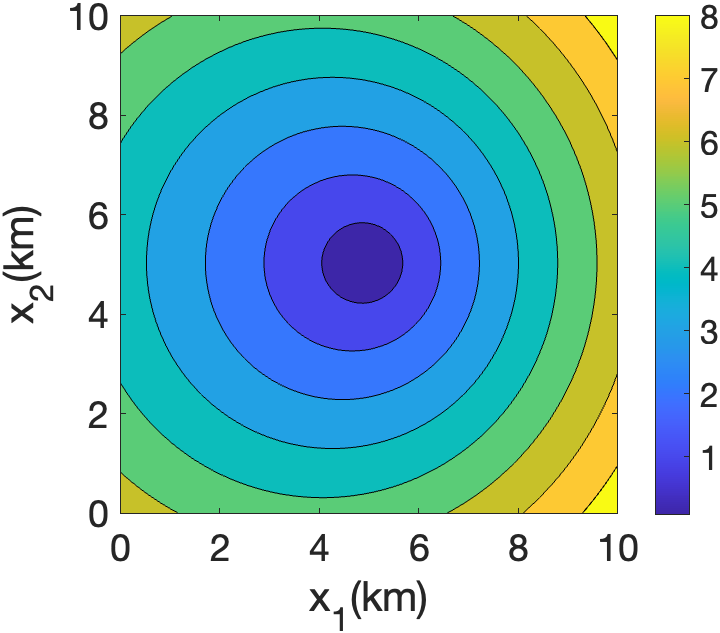}
         \caption{$T_4$}
         \label{fig:4_T4}
     \end{subfigure}
     \hfill
     \begin{subfigure}[b]{0.3\textwidth}
         \centering
         \includegraphics[width=\textwidth]{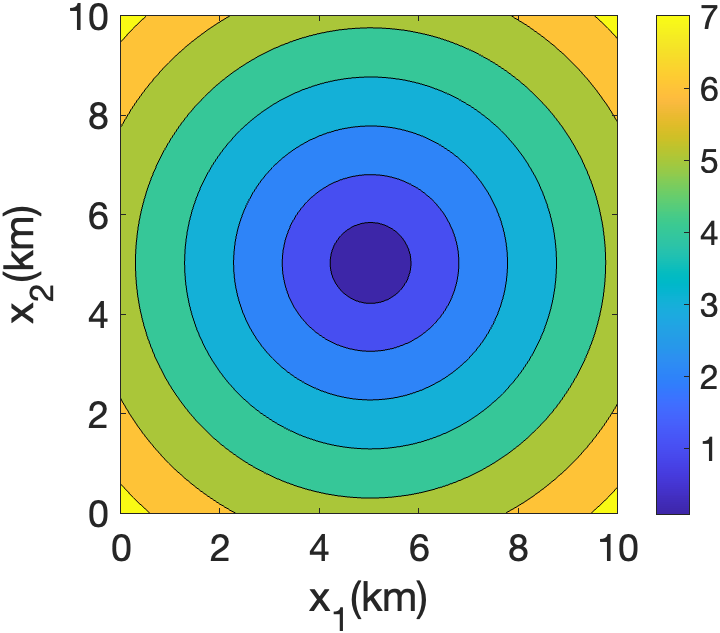}
         \caption{$T_5$}
         \label{fig:4_T5}
     \end{subfigure}
        \caption{Example 4a: The contour plots for the solutions $u$, $T_1$, $T_2$, $T_3$, $T_4$, and $T_5$ in $10^3$ seconds for 5-member ensemble of constant currents.}
        \label{fig:4_contours}
\end{figure}

\begin{figure}[htbp!]
    \centering
    \includegraphics[scale = 0.16]{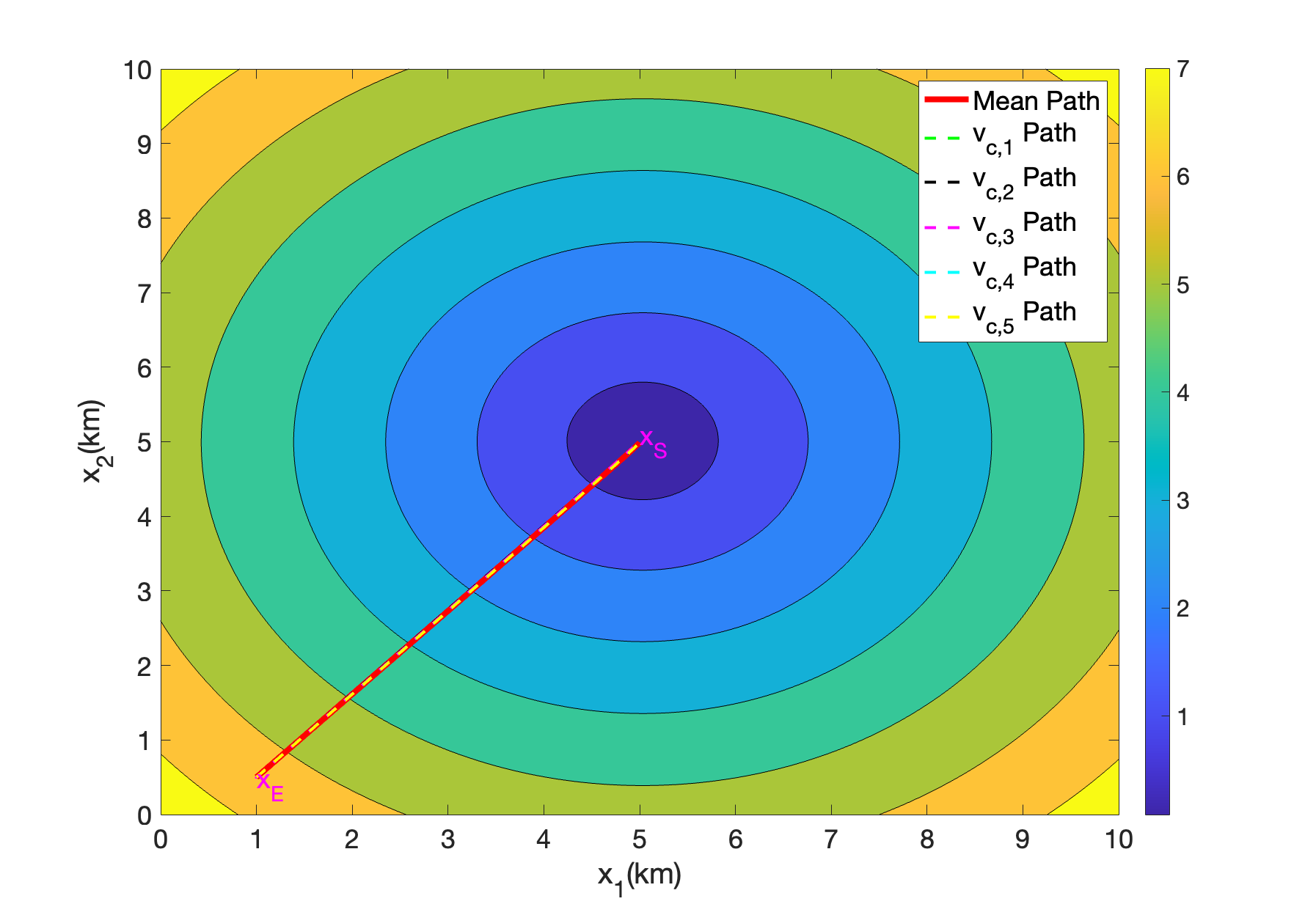}
    \caption{Example 4a: Mean uncertainty path (red) and deterministic paths on top of the contour of $u$ in $10^3$ seconds considering 5 ocean current models.}
    \label{fig:5_model}
\end{figure}

For our first case, we consider an ensemble of five constant ocean current models that exhibit symmetry.  The ensemble members are 
\begin{multline*}
\bv_{c,1}(\bx,t)=[0, -0.2]^T, \quad
\bv_{c,2}(\bx,t)=[0,0.2]^T, \quad  
\bv_{c,3}(\bx,t)=[0.2,0]^T, \\
\bv_{c,4}(\bx,t)=[-0.2,  0], \quad
\bv_{c,5}(\bx,t)=[0,0]^T,
\end{multline*}
where $p(1)=p(2)=p(3)=p(4)=p(5)=0.2$.  
The ensemble members satisfy 
$\bv_{c,1}+\bv_{c,2}+\bv_{c,3}+\bv_{c,4}+\bv_{c,5} = \bf{0}.$
In addition, it can be shown that
$$
u(\bx,\bx_S) = \sum _{i=1}^5 p(i)\cdot T_i(\bx,\bx_S,\gamma^*)
$$
following the definitions in \eqref{eq:weighted-time} and \eqref{eq:u-def}.  Combining the above results and linearizing the individual reachability times $T_i$ about $\bv_c \equiv 0$ (this is reasonable since the ocean speed of each ensemble member is small relative to the vehicle's maximum speed), suggests that the optimal mean reachability time $u$ should approximate the time required to travel in the case of zero ocean current.  Furthermore, the optimal mean-time path is guaranteed to be a straight line, since the straight-line path is optimal for each of the ensemble members.  Note that for sufficiently large ocean currents relative to the vehicle's maximum speed, the effects from the individual currents may not cancel out numerically.

Figures \ref{fig:4_contours} and \ref{fig:5_model} confirm the above observations.  Figure  \ref{fig:4_contours} illustrates the contours of the reachability functions $u$ and the individual ensemble members; the contours of $u$ appear identical to those of the zero-current case, $T_5$.  Figure \ref{fig:5_model} plots the optimal mean-time path $\gamma$ along with the optimal deterministic path for each ensemble member; all six paths appear to lie on the same straight line.

\subsubsection{Example 4b: Ensemble members exhibit greater heterogeneity}

We now consider a more complicated ensemble exhibiting spatial heterogeneity.  The ensemble members are shown in Figure \ref{fig:vcsubdividedquiver}; we set $p(1)=p(2)=p(3)=p(4)=p(5)=0.2$.

\begin{figure}[!t]
     \centering
     \begin{subfigure}[b]{0.31\textwidth}
         \centering
         \includegraphics[width=\textwidth]{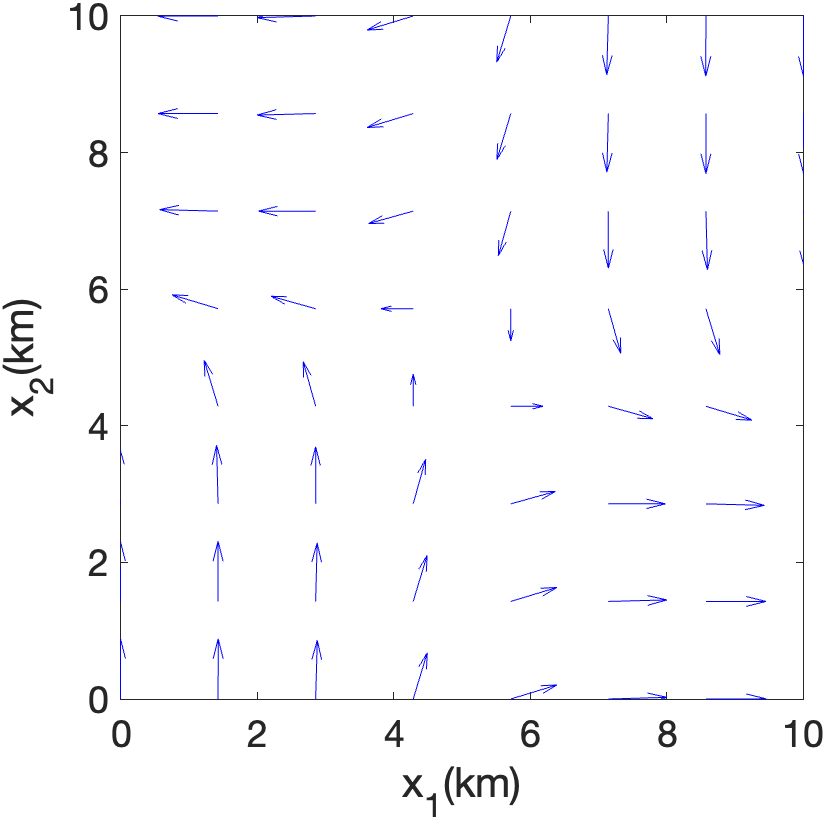}
         \caption{$\bv_{c,1}$}
         \label{fig:4b_v_c1}
     \end{subfigure}
     %\hfill
     \begin{subfigure}[b]{0.31\textwidth}
         \centering
         \includegraphics[width=\textwidth]{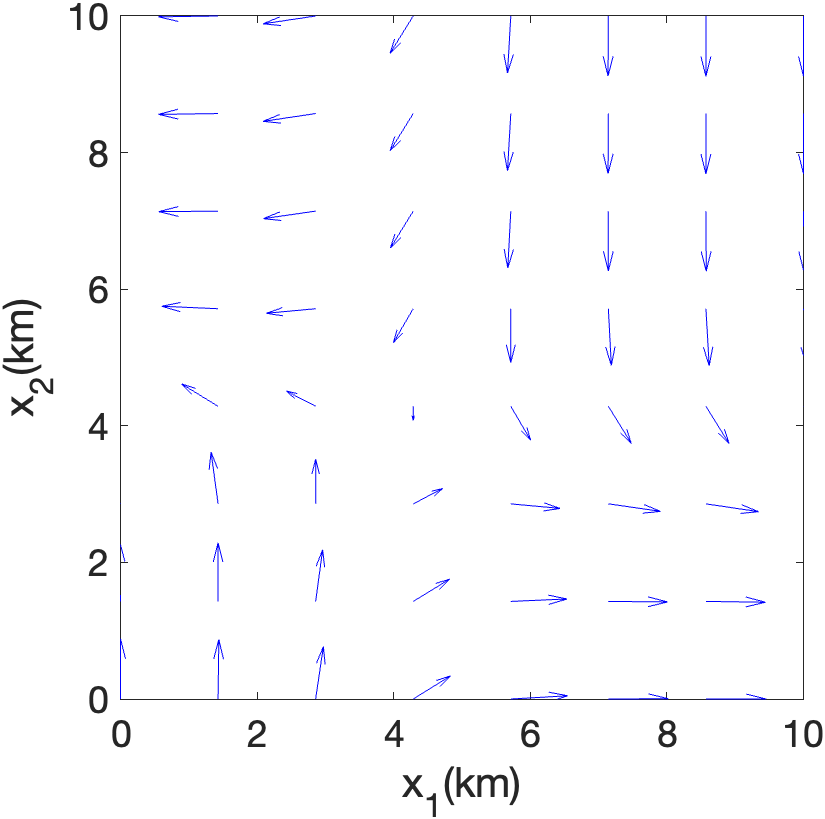}
         \caption{$\bv_{c,2}$}
         \label{fig:4b_v_c2}
     \end{subfigure}
     %\hfill
     \begin{subfigure}[b]{0.31\textwidth}
         \centering
         \includegraphics[width=\textwidth]{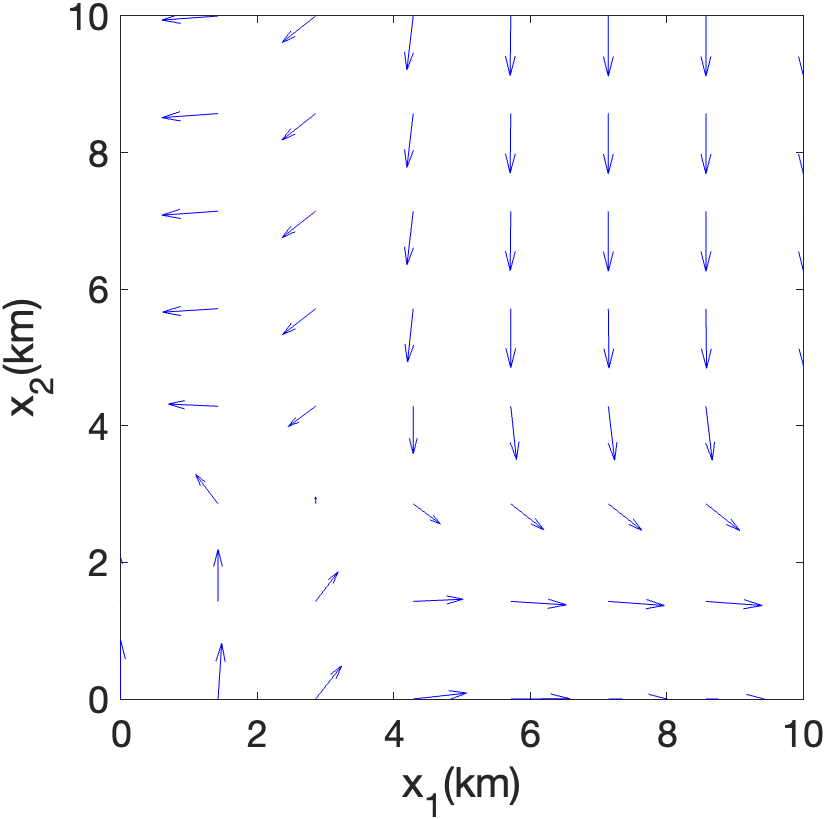}
         \caption{$\bv_{c,3}$}
         \label{fig:4b_v_c3}
     \end{subfigure}
     %\hfill
     \begin{subfigure}[b]{0.31\textwidth}
         \centering
         \includegraphics[width=\textwidth]{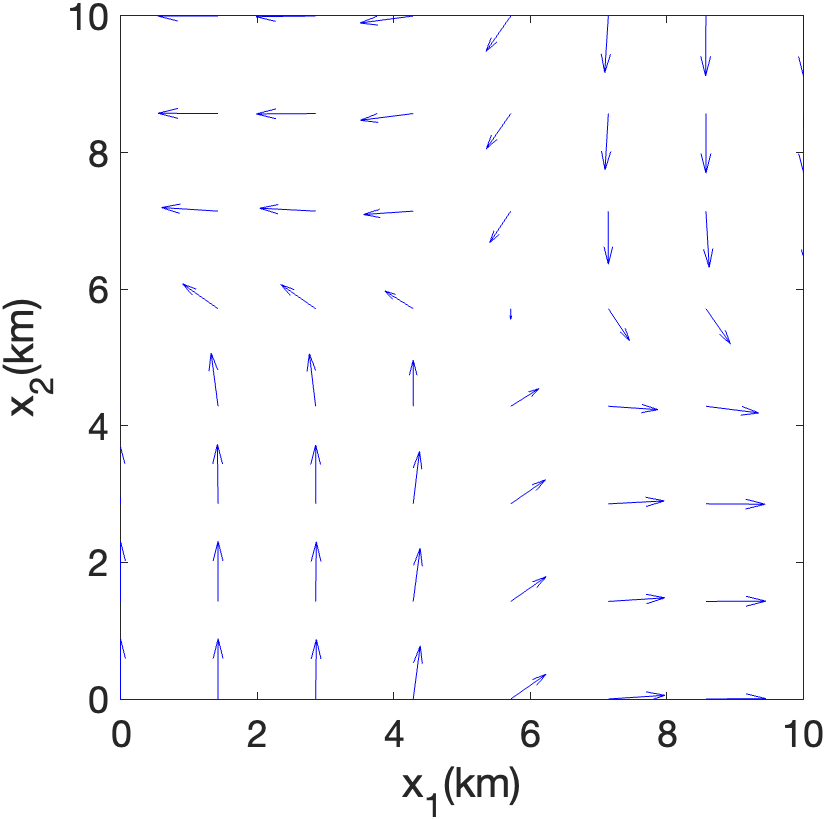}
         \caption{$\bv_{c,4}$}
         \label{fig:4b_v_c4}
     \end{subfigure}
     \begin{subfigure}[b]{0.31\textwidth}
         \centering
         \includegraphics[width=\textwidth]{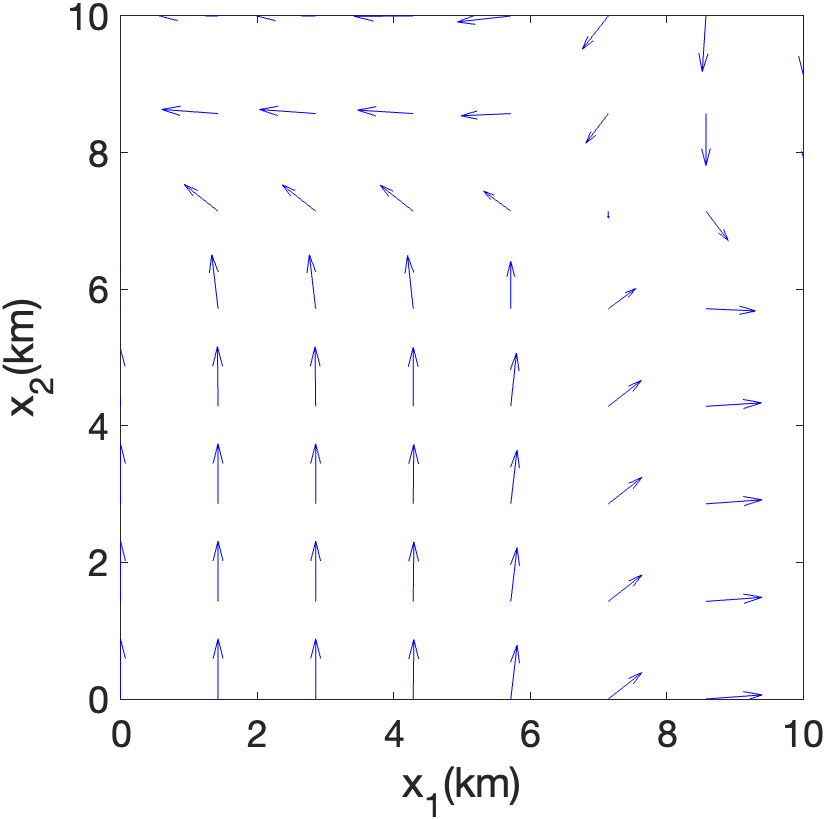}
         \caption{$\bv_{c,5}$}
         \label{fig:4b_v_c5}
     \end{subfigure}
\caption{Example 4b: Quiver plot of perturbed ocean current models over the same spatial domain $\Omega$.}
\label{fig:vcsubdividedquiver}
\end{figure}

\begin{figure}
     \centering
     \begin{subfigure}[b]{0.3\textwidth}
         \centering
         \includegraphics[width=\textwidth]{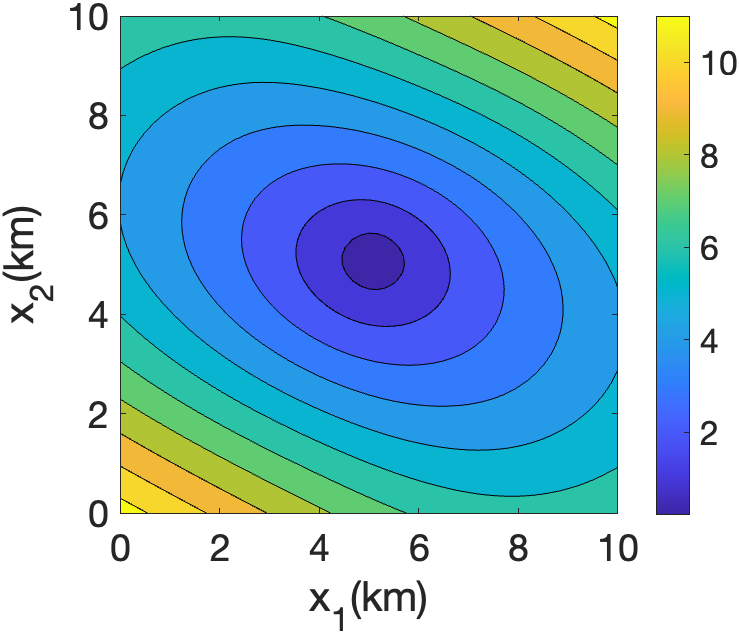}
         \caption{$u$}
         \label{fig:4b_u}
     \end{subfigure}
     \hfill
     \begin{subfigure}[b]{0.3\textwidth}
         \centering
         \includegraphics[width=\textwidth]{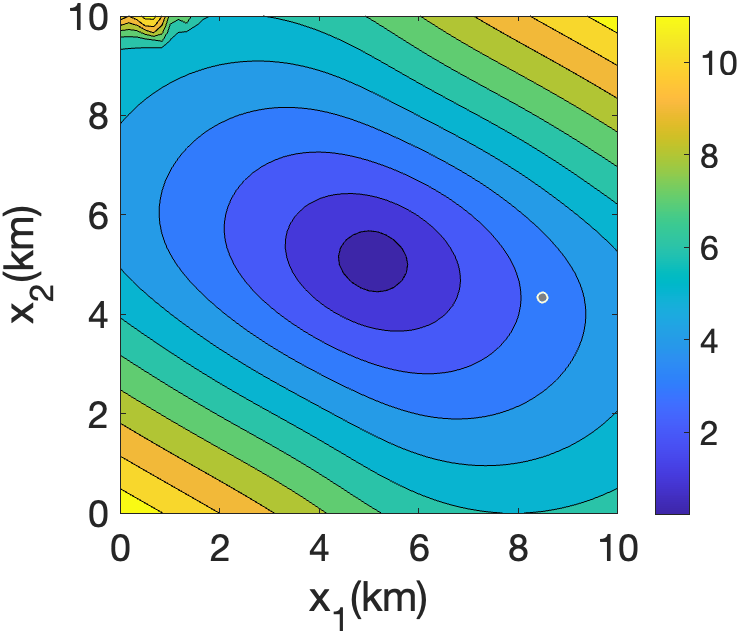}
         \caption{$T_1$}
         \label{fig:4b_T1}
     \end{subfigure}
     \hfill
     \begin{subfigure}[b]{0.3\textwidth}
         \centering
         \includegraphics[width=\textwidth]{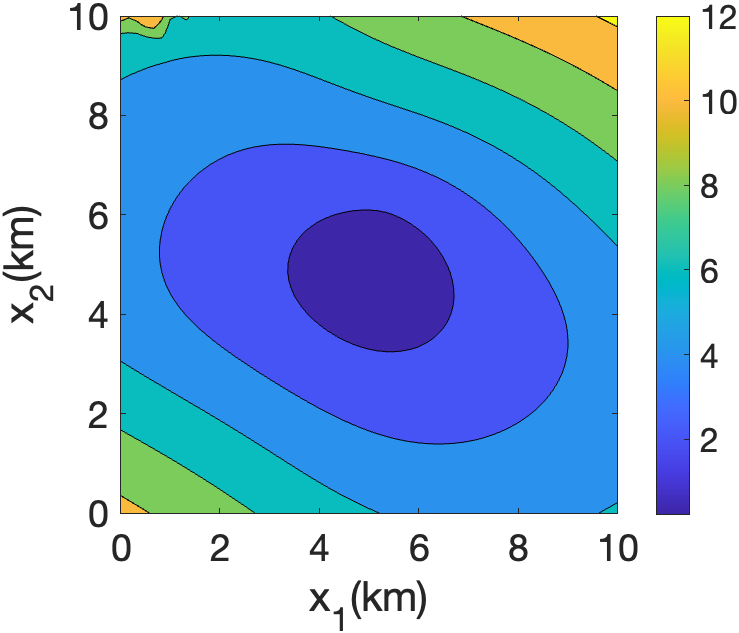}
         \caption{$T_2$}
         \label{fig:4b_T2}
     \end{subfigure}
     \begin{subfigure}[b]{0.3\textwidth}
         \centering
         \includegraphics[width=\textwidth]{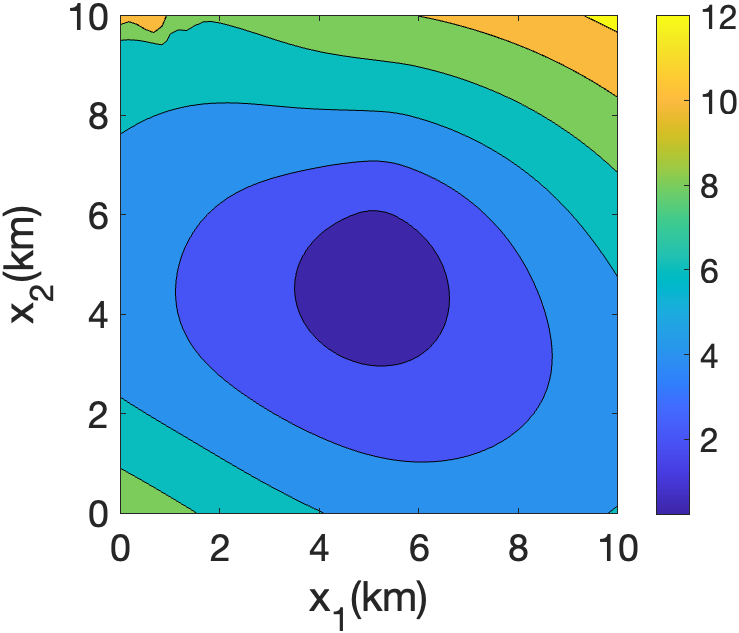}
         \caption{$T_3$}
         \label{fig:4b_T3}
     \end{subfigure}
     \hfill
     \begin{subfigure}[b]{0.3\textwidth}
         \centering
         \includegraphics[width=\textwidth]{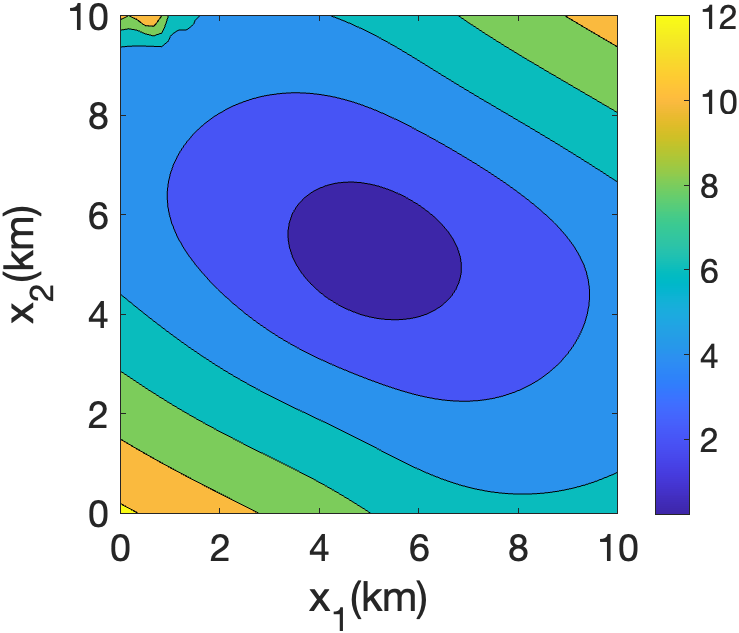}
         \caption{$T_4$}
         \label{fig:4b_T4}
     \end{subfigure}
     \hfill
     \begin{subfigure}[b]{0.3\textwidth}
         \centering
         \includegraphics[width=\textwidth]{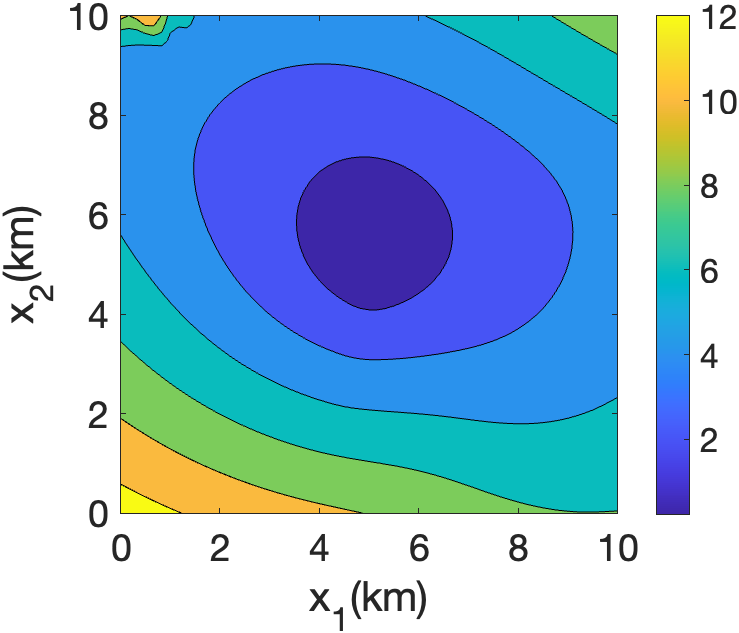}
         \caption{$T_5$}
         \label{fig:4b_T5}
     \end{subfigure}
        \caption{Example 4b: The contour plots for the solutions $u$, $T_1$, $T_2$, $T_3$, $T_4$, and $T_5$ in $10^3$ seconds for 5-member ensemble with greater heterogeneity. }
        \label{fig:4b_contours}
\end{figure}

\begin{figure}[htbp!]
    \centering
    \includegraphics[scale = 0.2]{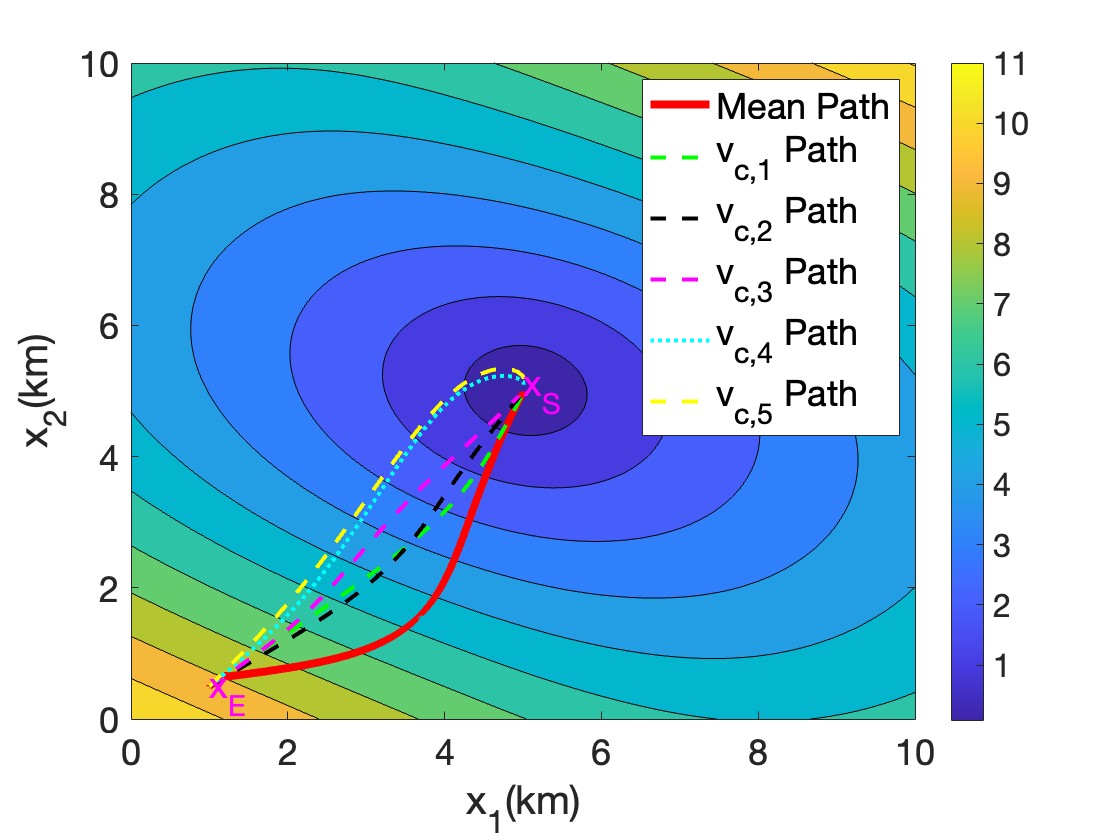}
    \caption{Example 4b: Mean uncertainty path (red) and five deterministic paths each with respect to only one ocean current model $\bv_{c,i}$ for $i=1,\cdots,5$ on top of the contour of $u$ in $10^3$ seconds considering 5 ocean current models with greater heterogeneity.}
    \label{fig:5_model_subdivided}
\end{figure}

Figure \ref{fig:4b_contours} presents the contour plots of the reachability times for this experiment; Figure \ref{fig:5_model_subdivided} illustrates the optimal mean path $\gamma^*$, along with the optimal path for each ensemble member.  As expected, the optimal mean path $\gamma^*$ tends to deviate most significantly from the paths of the individual ensemble members in those regions where the individual reachability times $T_i$ vary most significantly from one another.

These examples illustrate the importance of incorporating ensemble-based path planning, particularly when the ocean current environment has highly nonuniform currents throughout the domain of interest as the optimal path under uncertainty is nontrivial to ascertain.

\subsection{Example 5: A Realistic Ocean Model Application}

In our final example, we consider a more practical example that represents a realistic model ocean current.  Unsteady, periodically varying, time dependent double-gyre flows are observable in the ocean \cite{Shadden2005, Wang2016}.  Here, much like the vortex examples, while the general structure of the current is known, the uncertainty in the $i$-th ocean model lies in the strength of the ocean current, $\xi_{1,i}$, and the location of the gyre, $\xi_{2,i}$ and $\xi_{3,i}$.  We will consider a two-member ensemble of ocean models which are defined as

\begin{align}
    &\bv_{c,i}(\bx,t) = \begin{bmatrix}
    -\lambda \xi_{1,i} \sin(\pi f(x+\delta \xi_{2,i},t))*\cos(\pi(y+\delta \xi_{3,i}))\\
    \lambda \xi_{1,i} \cos(\pi f(x+\delta \xi_{2,i},t))*\sin(\pi(y+\delta \xi_{3,i})) *(2a(t)(x + \delta \xi_{2,i}) + b(t))
    \end{bmatrix}, \\
    &f(x,t) = a(t)x^2 + b(t)x, \\
    &a(t) = \sigma \sin(\omega t), \\
    &b(t) = 1 - 2\sigma \sin (\omega t),
\end{align}
where $\lambda = 0.5$ affects the current strength, $\delta=0.5$ affects the gyre location, $\sigma=0.05$ affects the steadiness where $\sigma=0$ would be a steady flow, and $\omega = \frac{\pi}{5}$ affects gyre period of motion.  We randomly and independently generate $\xi_{1,i}$, $\xi_{1,2}$, and $\xi_{1,3}$ from the uniform distribution $[0,1]$.  The probability of realization is set to be $p(1)=p(2)=0.5$ for the ensemble members.

We maintain the previous computational domain of $\Omega = [0\text{ km},10\text{ km}] \times [0\text{ km},10\text{ km}]$; the starting and target positions are $\bx_S = (5\text{ km},5\text{ km})$ and $\bx_E = (1\text{ km},0.5\text{ km})$, respectively. We set $\eta=(\eta_{x_1},\eta_{x_2})=(1.5,1.5)$.

\begin{figure}[t!]
     \centering
     \begin{subfigure}[b]{0.3\textwidth}
         \centering
         \includegraphics[width=\textwidth]{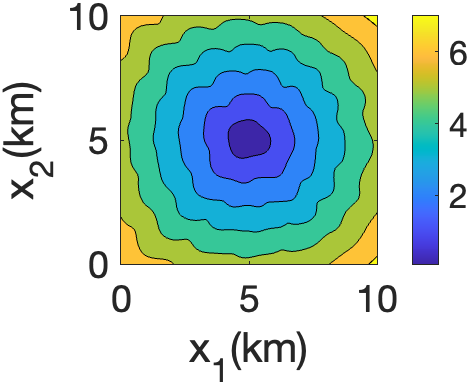}
         \caption{$u$}
         \label{fig:5_u}
     \end{subfigure}
     \hfill
     \begin{subfigure}[b]{0.3\textwidth}
         \centering
         \includegraphics[width=\textwidth]{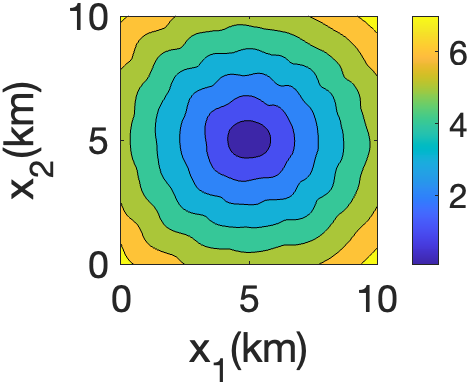}
         \caption{$T_1$}
         \label{fig:5_T1}
     \end{subfigure}
     \hfill
     \begin{subfigure}[b]{0.3\textwidth}
         \centering
         \includegraphics[width=\textwidth]{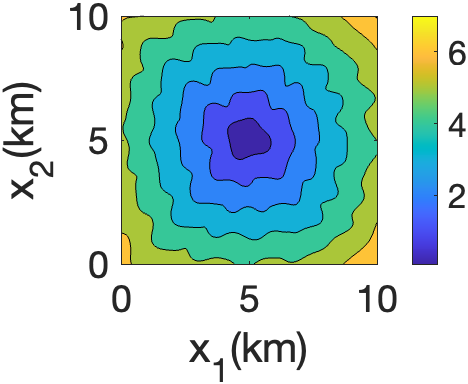}
         \caption{$T_2$}
         \label{fig:5_T2}
     \end{subfigure}
        \caption{Example 5: Contour plots of the reachability times $u$, $T_1$, and $T_2$ in $10^3$ seconds.}
        \label{fig:5_contours}
\end{figure}

\begin{figure}[t!]
    \centering
    \includegraphics[scale = 0.2]{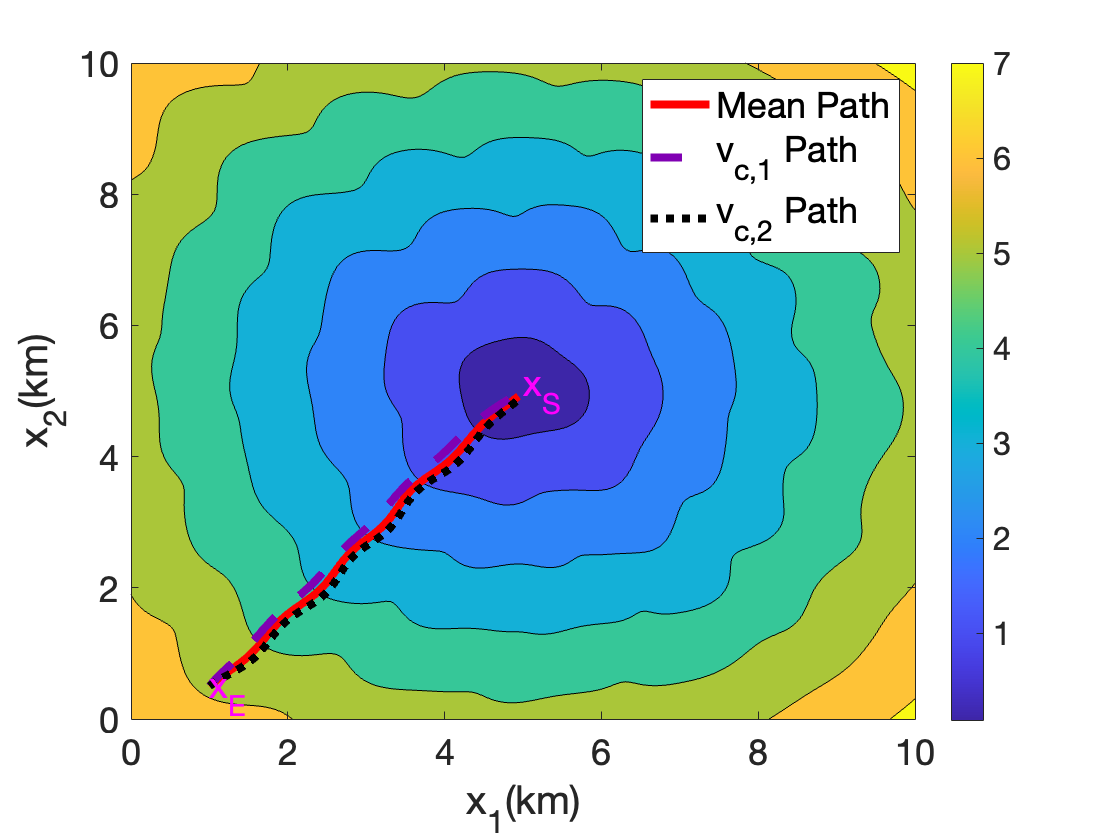}
    \caption{Example 5: Two deterministic paths (purple and black) and the mean uncertainty path (red) from $(5\text{ km},5\text{ km})$ to $(1\text{ km},0.5\text{ km})$ on top of contour of $u$ in $10^3$ seconds.}
    \label{fig:double_gyre_paths}
\end{figure}

Choosing the parameters as above, we assume that the uncertainties of the model are significant but controlled.  This is reflected in Figure \ref{fig:5_contours} where we see $T_1$ and $T_2$ are clearly distinct, but are still reasonably similar.  As previously seen, when $T_1 \approx T_2$, we see that $u$ reflects characteristics of both $T_1$ and $T_2$.  Figure \ref{fig:double_gyre_paths} mirrors this in that the deterministic paths constructed are different, but close in trajectory with the mean uncertainty path situated between them.

\section{Conclusion}
\label{sec:5}
In this work, we extended the deterministic path planning framework~\cite{Brandman2023} to account for uncertainty in the ocean current forecast.  We accounted for this uncertainty through an ensemble of ocean velocity models; using this, we derived a novel system of Hamilton-Jacobi PDEs that enable us to determine the minimum mean reachability time and the associated optimal paths.  In order to solve this PDE system, we developed an extension of the fast sweeping method to systems of PDEs that alternates PDE solves within a single sweep and demonstrated its superior performance relative to an alternative sweeping strategy.  We %also 
demonstrated convergence of our method and illustrated, through several examples, that the resulting optimal mean-time trajectory can differ significantly from the optimal trajectories associated with individual ensemble members.  

Looking ahead, we are interested in extending this work in several directions.  %From an applications point of view, one potential area to explore would be the uncertainty of the location of the target object that is advected by the unknown ocean current, specifically for time-dependent moving targets.  Combining the optimal path construction with data assimilation techniques is a possibility in reducing the target uncertainty.  
%The observed difference in behavior between $T_1 \approx T_2$ and $\frac{T_1}{T_2}>>1$ warrants the consideration of a deviation measure to understand how distinct $T_1$ must be from $T_2$ such that we observe a very distinct mean optimal path.  %Another interesting application would be to consider %cases where the vehicle's maximum speed is smaller %than the ocean current.  
From an algorithmic perspective, there are two natural directions to consider.  First, it seems natural to insert obstacles into the computational domain.  Second, we would like to remove the assumption that each point in the domain is reachable.  
%In addition, 
%we are interested in quantifying the extent to which adaptive mesh refinement of subdomains where high vorticity %exists can increase the accuracy of solutions while %keeping the computational cost approximately fixed.

From an analysis perspective, several questions remain.  First, it would be of great interest to extend the existence and uniqueness of viscosity solutions resulting from a single Hamilton-Jacobi equation to solutions resulting from a system of Hamilton-Jacobi equations.  This would provide insight regarding convergence for fast sweeping applied to systems.  Second, the difference in behavior between $T_1 \approx T_2$ and $\frac{T_1}{T_2}>>1$, observed in example 3b, needs further investigation.  In particular, we would like to quantify how deviation between the $T_i$ impacts the mean optimal path.

\section*{Acknowledgment}

This research was supported by the U.S. Naval Research Laboratory under a 6.2 Base Program (62A1G1), with support for J. Valyou provided under the Naval Research Enterprise Internship Program (NREIP).

\section*{Appendix A. Linear-in-time constant current analytical solution}
\label{sec:app}

We derive the analytical solution for linear-in-time constant currents.

Take a current of the form 
\begin{equation}
    v_c(\bx,t)=t\tilde{\bv}_c + \hat{\bv}_c
    \label{eq:linearintime}
\end{equation}

where $\tilde{\bv}_c$ and $\hat{\bv}_c$ are constant currents.  Without loss of generality, take $\bx_S=(0,0)$ and $\bx_E=(l,0)$ where $l$ is a constant.

Let $\bv_g = \bv_w + \bv_c$ where $\bv_g$ is the vehicle velocity relative to the ground.

Since there is no movement in the second direction component,

\begin{equation}
    \int_{0}^{T^*} v_{w}^{(2)} \,dt = - \int_{0}^{T^*} v_{c}^{(2)} \,dt.
    \label{eq:inteq}
\end{equation}

Given the conditions in \eqref{eqn:ode_condition}, we can write the maximum vehicle velocity as

\begin{equation}
    v_{w}^{(1)} = \sqrt{s_{max}^2 - {v_{w}^{(2)}}^2}.\end{equation}
    
Therefore, the following represents movement in the first direction component:

\begin{equation}
    l = \int_{0}^{T^*} \sqrt{s_{max}^2 - {v_{w}^{(2)}}^2} + v_{c}^{(1)} \,dt.
\end{equation}

Assuming $T^*>0$ and normalizing by $T^*$,

\begin{equation}
    \frac{l}{T^*} = \frac{\int_{0}^{T^*} \sqrt{s_{max}^2 - {v_{w}^{(2)}}^2} + v_{c}^{(1)} \,dt}{T^*}.
\end{equation}

Applying the definition in \eqref{eq:linearintime},
\begin{align}
    \frac{l}{T^*} &= \frac{\int_{0}^{T^*} \sqrt{s_{max}^2 - {v_{w}^{(2)}}^2} + t\tilde{v}_{c}^{(1)} + \hat{v}_{c}^{(1)} \,dt}{T^*}\\
    &= \frac{\int_{0}^{T^*} \sqrt{s_{max}^2 - {v_{w}^{(2)}}^2} \,dt}{T^*} + \frac{T^*\tilde{v}_{c}^{(1)}}{2} + \hat{v}_{c}^{(1)}.
\end{align}

By Jensen's Inequality,

\begin{equation}
    \frac{l}{T^*} \leq \sqrt{s_{max}^2 - \left(\frac{\int_{0}^{T^*} v_{w}^{(2)} \,dt}{T^*}\right)^2} + \frac{T^*\tilde{v}_{c}^{(1)}}{2} + \hat{v}_{c}^{(1)}.
\end{equation}

Applying the definition in \eqref{eq:inteq},

\begin{equation}
    \frac{l}{T^*} \leq \sqrt{s_{max}^2 - \left(\frac{-\int_{0}^{T^*} v_{c}^{(2)} \,dt}{T^*}\right)^2} + \frac{T^*\tilde{v}_{c}^{(1)}}{2} + \hat{v}_{c}^{(1)}.
\end{equation}

Applying the definition in \eqref{eq:linearintime} again,

\begin{align}
    \frac{l}{T^*} &\leq \sqrt{s_{max}^2 - \left(\frac{-\int_{0}^{T^*} t\tilde{v}_{c}^{(2)} + \hat{v}_{c}^{(2)} \,dt}{T^*}\right)^2} + \frac{T^*\tilde{v}_{c}^{(1)}}{2} + \hat{v}_{c}^{(1)}\\
    &= \sqrt{s_{max}^2 - \left(-\frac{T^*\tilde{v}_c^{(2)}}{2} - \hat{v}_{c}^{(2)}\right)^2} + \frac{T^*\tilde{v}_{c}^{(1)}}{2} + \hat{v}_{c}^{(1)}\\
    &= \sqrt{s_{max}^2 - \frac{\left(T^* \tilde{v}_c^{(2)}\right)^2}{4} + T^* \tilde{v}_c^{(2)} \hat{v}_c^{(2)} - \left(\hat{v}_c^{(2)}\right)^2 } + \frac{T^* \tilde{v}_c^{(1)}}{2} + \hat{v}_c^{(1)}.
\end{align}

Multiplying each side by $T^*$,

\begin{equation}
    l \leq T^*\sqrt{s_{max}^2 - \frac{\left(T^* \tilde{v}_c^{(2)}\right)^2}{4} + T^* \tilde{v}_c^{(2)} \hat{v}_c^{(2)} - \left(\hat{v}_c^{(2)}\right)^2 } + \frac{T^* \tilde{v}_c^{(1)}}{2} + \hat{v}_c^{(1)}.
\end{equation}

If the optimal vehicle velocity is chosen, then this equation becomes equality giving an implicit analytical solution

\begin{equation}
    T^* = \frac{l}{\sqrt{s_{max}^2 - \frac{\left(T^* \tilde{v}_c^{(2)}\right)^2}{4} + T^* \tilde{v}_c^{(2)} \hat{v}_c^{(2)} - \left(\hat{v}_c^{(2)}\right)^2 } + \frac{T^* \tilde{v}_c^{(1)}}{2} + \hat{v}_c^{(1)}}.
\end{equation}


\bibliography{SpringerJournal/sn-bibliography}% common bib file

\end{document}